\documentclass[10pt]{article}
\usepackage[utf8]{inputenc} 
\usepackage[english]{babel}  
\usepackage{amssymb,amsmath,amsthm,amsfonts} 
\numberwithin{equation}{section}
\usepackage{mathtools}
\usepackage{graphicx}
\usepackage{hyperref} 
\mathtoolsset{showonlyrefs}
\usepackage{color}

\newcommand*{\N}{\mathbb N}
\newcommand*{\R}{\mathbb R}

\newtheorem{theorem}{Theorem}[section]
\newtheorem{lemma}[theorem]{Lemma}
\newtheorem{proposition}[theorem]{Proposition}
\newtheorem{corollary}[theorem]{Corollary}
\newtheorem{remark}[theorem]{Remark}

\let\originalleft\left
\let\originalright\right
\renewcommand{\left}{\mathopen{}\mathclose\bgroup\originalleft}
\renewcommand{\right}{\aftergroup\egroup\originalright}

\newlength{\leftstackrelawd}
\newlength{\leftstackrelbwd}
\def\leftstackrel#1#2{\settowidth{\leftstackrelawd}%
	{${{}^{#1}}$}\settowidth{\leftstackrelbwd}{$#2$}%
	\addtolength{\leftstackrelawd}{-\leftstackrelbwd}%
	\leavevmode\ifthenelse{\lengthtest{\leftstackrelawd>0pt}}%
	{\kern-.5\leftstackrelawd}{}\mathrel{\mathop{#2}\limits^{#1}}}

\newcommand{\brc}[1]{  \left\{#1\right\} } 
  
\newcommand{\norm}[1]{  \left\|#1\right\| } 
\newcommand{\pare}[1]{\left(#1\right)}    
\newcommand{\cor}[1]{  \left[#1\right] }      
\newcommand{\abs}[1]{  \left\vert#1\right\vert }     
     
\newcommand{\bb}[1]{\mathbb{#1}}

\def\mathcolor#1#{\@mathcolor{#1}}
\def\@mathcolor#1#2#3{%
	\protect\leavevmode
	\begingroup
	\color#1{#2}#3%
	\endgroup
}

\def\11{\textbbm{1}}

\def\vint_#1{\mathchoice%
          {\mathop{\kern 0.2em\vrule width 0.6em height 0.69678ex depth -0.58065ex
                  \kern -0.8em \intop}\nolimits_{\kern -0.4em#1}}%
          {\mathop{\kern 0.1em\vrule width 0.5em height 0.69678ex depth -0.60387ex
                  \kern -0.6em \intop}\nolimits_{#1}}%
          {\mathop{\kern 0.1em\vrule width 0.5em height 0.69678ex
              depth -0.60387ex
                  \kern -0.6em \intop}\nolimits_{#1}}%
          {\mathop{\kern 0.1em\vrule width 0.5em height 0.69678ex depth -0.60387ex
                  \kern -0.6em \intop}\nolimits_{#1}}}
\def\vintslides_#1{\mathchoice%
          {\mathop{\kern 0.1em\vrule width 0.5em height 0.697ex depth -0.581ex
                  \kern -0.6em \intop}\nolimits_{\kern -0.4em#1}}%
          {\mathop{\kern 0.1em\vrule width 0.3em height 0.697ex depth -0.604ex
                  \kern -0.4em \intop}\nolimits_{#1}}%
          {\mathop{\kern 0.1em\vrule width 0.3em height 0.697ex depth -0.604ex
                  \kern -0.4em \intop}\nolimits_{#1}}%
          {\mathop{\kern 0.1em\vrule width 0.3em height 0.697ex depth -0.604ex
                  \kern -0.4em \intop}\nolimits_{#1}}}

\newcommand{\aveint}[2]{\mathchoice%
          {\mathop{\kern 0.2em\vrule width 0.6em height 0.69678ex depth -0.58065ex
                  \kern -0.8em \intop}\nolimits_{\kern -0.45em#1}^{#2}}%
          {\mathop{\kern 0.1em\vrule width 0.5em height 0.69678ex depth -0.60387ex
                  \kern -0.6em \intop}\nolimits_{#1}^{#2}}%
          {\mathop{\kern 0.1em\vrule width 0.5em height 0.69678ex depth -0.60387ex
                  \kern -0.6em \intop}\nolimits_{#1}^{#2}}%
          {\mathop{\kern 0.1em\vrule width 0.5em height 0.69678ex depth -0.60387ex
                  \kern -0.6em \intop}\nolimits_{#1}^{#2}}}

\def\nkj{n_{k_j}}
\def\nk{n_{k}}

\begin{document}
	
	\title{\bf Homogenization for nonlocal evolution problems with three different smooth kernels}
	
	\author{Monia Capanna, Jean C. Nakasato, Marcone C. Pereira and
		Julio D. Rossi}
		
		\date{}
	
	\maketitle

			\date{\today}


	\begin{abstract} 
		In this paper we consider the homogenization of the evolution problem associated 
		with a jump process that involves three different smooth kernels that govern the jumps to/from different parts of the domain. 
		We assume that the spacial domain 
		is divided into a sequence of two subdomains $A_n \cup B_n$ and we have three different smooth
		kernels, one that controls the jumps from $A_n$ to $A_n$, a second one that controls
		the jumps from $B_n$ to $B_n$ and the third one that governs the interactions between $A_n$ and $B_n$.
		
		Assuming that $\chi_{A_n} (x) \to X(x)$ weakly in $L^\infty$ (and then $\chi_{B_n} (x) \to 1-X(x)$ weakly in $L^\infty$)
		as $n \to \infty$ and that the initial condition is given by a density $u_0$ in $L^2$ 
		we show that there is an homogenized limit system in which the three kernels and the limit function $X$ appear.
		 When the initial condition is a delta at one point, $\delta_{\bar{x}}$ (this corresponds to the process that starts at $\bar{x}$) we show that
		 there is convergence along subsequences such that $\bar{x} \in A_{n_j}$ or $\bar{x} \in B_{n_j}$ for every $n_j$
		 large enough. 
		 
		 We also provide
		a probabilistic interpretation of this evolution equation in terms of a stochastic process that
		describes the movement of a particle that jumps in $\Omega$ according to the three different kernels and show 
		that the underlying process converges in distribution to a limit process associated with the limit equation.  

		We focus our analysis in Neumann type boundary conditions and briefly describe at the end
		how to deal with Dirichlet boundary conditions. 	
		\end{abstract}
	
	{Keywords: heterogeneous media, homogenization, nonlocal equations, Neumann problem, Dirichlet problem.}

	\section{Introduction.}
	\label{Sect.intro}
	
	\setcounter{equation}{0}

	Our main goal in this paper is to study the homogenization that occurs when one deals with nonlocal
	evolution problems with 
	different non-singular kernels that act in different domains.
	This paper is a natural continuation of \cite{nosotros} where the stationary case was
	studied.

	Consider a partition of the ambient space $\overline{\Omega}$ (a bounded domain in $\mathbb{R}^N$) 
	into two subdomains $A$, $B$, and consider a nonlocal problem
	in which we have three different smooth kernels. One (that we call $J$) that measures the probability of jumping
	from $A$ to $A$ ($J(x,y)$ is the probability that a particle that is at $x\in A$ moves to $y\in A$), 
	another one ($G$) that is involved in jumps from $B$ to $B$ and a third one ($R$)
	that gives the interactions between $A$ and $B$. 
	Remark that the involved kernels can be of convolution type, that is, we could have
	for instance, $J(x,y) = J(x-y)$ (this special form of the kernels is often used in applications). 
	However, we only use in our arguments that the kernels $V =J$, $G$ and $R$ 
	are non-singular functions which satisfy the following hypotheses that will be assumed from now on
	$$
	{\bf (H)} \qquad 
	\begin{gathered}
	V \in \mathcal{C}(\R^N\times \R^N,\R) \textrm{ is non-negative with } V(x,x)>0, \;  \textrm{symmetric, } V(x,y) = V(y,x) \textrm{ for every $x,y \in \R^N$, and } \\
	\int_{\R^N} V(x,y) \, dy = 1.
	\end{gathered}
	$$

We take a sequence of partitions
$A_n$, $B_n$ of the fixed ambient space $\overline{\Omega}$ such that $\overline{\Omega} = A_n \cup B_n$, $A_n\cap B_n= \emptyset$,
$B_n$ is open, has a Lipchitz boundary (consequently $\abs{\partial B_n \cap \Omega}=\abs{\partial A_n \cap \Omega}=0$) and
\begin{equation} \label{cond.sets}
\begin{array}{l}
\bullet \  \chi_{A_n} (x) \rightharpoonup X(x), \qquad \mbox{ weakly in } L^\infty (\overline\Omega), 
 \\[10pt]
\bullet  \ \chi_{B_n} (x) \rightharpoonup 1-X(x) \qquad \mbox{ weakly in } L^\infty (\overline\Omega), \\[10pt]
\qquad \mbox{with } 0< X(x) <1.
\end{array}
\end{equation}

Associated to this sequence of partitions we consider the diffusion process that we describe next.
We want to analyze the evolution of a particle that moves in $ \overline\Omega$. To do that we introduce three 
families $\brc{E^1_k}_{k\in \bb N}$, $\brc{E^2_k}_{k\in \bb N}$ and $\brc{E^3_k}_{k\in \bb N}$ of independent 
random variables with exponential distribution of parameter $\frac{1}{3}$. Define
\begin{align*}
\Upsilon_k:=\min_{i\in\brc{1,2,3}}\brc{E_k^i},\qquad\forall k\in\mathbb N.
\end{align*} 
The set $\brc{\Upsilon_k}_{k\in\mathbb N}$ is a family of independent random variables distributed as an exponential of parameter $1$.
Fixing $\tau_0=0$, we define recursively the random times
\begin{align}\label{times}
\tau_k=\tau_{k-1}+\Upsilon_k, \qquad \forall k\in \mathbb N.
\end{align}
We denote by $Y_n\pare{t}$ the position of the particle at time $t$.
The evolution of the particle is described as follows: At the times $\{\tau_k\}$ the particle chooses a site $y\in {\overline\Omega}$ according to the kernels $J$, $R$ or $G$. The jumps from a site in $A_n$ to another site in $A_n$ are ruled by $J$,  the jumps between $A_n$ and $B_n$ (or vice versa) are ruled by $R$, the jumps from a site in $B_n$ to a site in $B_n$ are ruled by $G$. More precisely, if $\Upsilon_k=E_k^1$ the particle chooses a site $y\in \overline\Omega$ according to $J\pare{Y_n(\tau_{k-1}),y}$ and it jumps on it only if $Y_n(\tau_{k-1})\in A_n$ and $y\in A_n$ otherwise the particle remains in its current position. If $\Upsilon_k=E_k^2$ the particle chooses a site $y\in \overline\Omega$ according to the kernel $R\pare{Y_n(\tau_{k-1}),y}$ and it jumps on it only if $Y_n(\tau_{k-1})\in A_n$ and $y\in B_n$ (or if $Y_n(\tau_{k-1})\in B_n$ and $y \in A_n$). Finally, if $\Upsilon_k=E_k^3$ the particle chooses a site $y\in \overline\Omega$ according to $G\pare{Y_n(\tau_{k-1}),y}$ and it jumps on it only if $Y_n(\tau_{k-1})\in B_n$ and $y\in B_n$.

The process $Y_n(t)$ is a Markov process whose generator ${L}_n$ is defined on functions $f:\overline{\Omega}
\to \mathbb{R}$ with $f|_{A_n}\in C\pare{A_n}$ and $f|_{B_n}\in C\pare{B_n}$ as
\begin{align}\label{gen.intro}
{L}_n f(x)=
&\chi_{A_n}\pare{x}\int_{\Omega}\chi_{A_n}\pare{y} J\pare{x,y}\pare{f\pare{y}-f\pare{x}}dy+\chi_{B_n}\pare{x}\int_{\Omega}\chi_{B_n}\pare{y} G\pare{x,y}\pare{f\pare{y}-f\pare{x}}dy\\[6pt]
&+\chi_{A_n}\pare{x}\int_\Omega \chi_{B_n}\pare{y}R\pare{x,y}\pare{f\pare{y}-f\pare{x}}dy+\chi_{B_n}\pare{x}\int_\Omega \chi_{A_n}\pare{y}R\pare{x,y}\pare{f\pare{y}-f\pare{x}}dy.
\end{align}

With an initial distribution of the position of the particle at time $t=0$, $u_0$, 
the associated evolution problem (whose solution is the density of the process $Y_n$, see 
Corollary \ref{corol.densidad}) reads as
\begin{equation}\label{evol.intro}
\left\{
\begin{array}{ll}
\displaystyle \frac{\partial u_n}{\partial t} (t,x) = {L}_n u_n(t,x), \qquad &  \, t>0, \, x\in \overline{\Omega}, \\[6pt]
u_n(0,x)=u_0(x), \qquad & x\in \overline\Omega.
\end{array}
\right.
\end{equation}

Notice that we have an evolution equation of Neumann type since the particle remains inside $\overline{\Omega}$
for every positive time (there are no particles entering or leaving the domain). For this evolution the total mass is preserved in time, that is,
$$
\int_\Omega u_n (t,x) \, dx = \int_\Omega u_0 (x) \, dx \qquad \forall t>0,
$$
as is expected for a Neumann problem. We will focus in this case, but at the end of this paper we
 will briefly comment on Dirichlet type problems (in this case the particle is allowed to jump outside $\Omega$ and is killed when doing so).

Our goal is to take the limit, as $n\to +\infty$, both in the processes $Y_n(t)$ and in the associated densities $u_n(t,x)$. 
To this end we need to look at the process $Y_n(t)$ as a couple $\pare{Y_n\pare{t}, I_n\pare{t}}$. In our notation $I_n(t)$ contains explicitly the information over the set ($A_n$ or $B_n$) in which $Y_n(t)$ is located. More precisely, $I_n\pare{t}=1$ (or $2$) if the particle is in $A_n$ (or in $B_n$ respectively)
at time $t$.

First, we assume that the initial position $Y_n(0)$ is described in terms of a given distribution $u_0 \in L^2(\Omega)$. We suppose that
\begin{align}
P\pare{Y_n\pare{0}\in E}=\int_{E} u_0(z) \, dz, \end{align}
for every measurable set $E\subseteq \overline\Omega$.

\begin{theorem} \label{teo.main.intro}
Let the initial condition be given by a distribution $u_0 \in L^2(\Omega)$.
Assume \eqref{cond.sets} and fix $T>0$. We have that, as $n\to \infty$,
\begin{equation}\label{limite.debil.u.intro}
\begin{array}{l}
\displaystyle u_n(t,x) \rightharpoonup u(t,x),\qquad \mbox{weakly in } L^2 ( (0,T) \times \Omega ), \\[5pt]
\displaystyle \chi_{A_n} (x) u_n(t,x) \rightharpoonup a(t,x),\qquad \mbox{weakly in } L^2 ( (0,T)\times \Omega ), \\[5pt]
\displaystyle \chi_{B_n} (x) u_n(t,x) \rightharpoonup b(t,x),\qquad \mbox{weakly in } L^2 ((0,T) \times \Omega ). 
\end{array}
\end{equation}
These limits verify
$$
u(t,x) = a(t,x) + b(t,x)
$$
and are characterized by
the fact that $(a,b)$ is the unique solution to the following system,
\begin{equation}\label{sys1.intro}
\left\{
\begin{array}{ll}
\displaystyle \frac{\partial a}{\partial t}\pare{ t,x}=\int_{\Omega}J\pare{x, y}\pare{X(x)a\pare{t,y}-X(y)a\pare{t,x}}\, dy
\\[10pt]
\displaystyle \qquad \qquad \quad 
+\int_{\Omega}R\pare{x, y}\pare{X(x)b(t,y)-\pare{1-X(y)}a(t,x)}\, dy \qquad & \, t>0,\, x\in \Omega , \\[10pt]
\displaystyle \frac{\partial b}{\partial t}\pare{t,x}=\int_\Omega G\pare{x, y}\cor{\pare{1-X(x)}b\pare{t,y}-\pare{1-X(y)}b\pare{t,y}}dy\\[10pt]
\qquad\qquad\quad +\displaystyle\int_\Omega R(x,y)\cor{\pare{1-X(x)}a(t,y)-X(y)b(t,x)}dy\, dy
\qquad & \, t>0, \,  x\in \Omega , \\[10pt]
a\pare{0,x}=X\pare{x}u_0\pare{x}, \quad b\pare{ 0,x}=\pare{1-X(x)}u_0\pare{x}
\qquad & x\in \overline\Omega.
\end{array} \right.
\end{equation}
\end{theorem}

Remark that in the limit we obtain a system rather than a single equation. However, as we show here, 
there is uniqueness for solutions to the limit system and hence this characterizes the limit $u(t,x) = a(t,x) + b(t,x)$.

Before stating the next theorem that describes the limit distribution of our stocastic process, we introduce some notation. Given a metric space {$\mathcal{X}$}, for $T>0$, we denote by $D\pare{[0, T], {\mathcal{X}}}$ the space of all trajectories cadlag defined in $[0, T]$ and taking values in 
${\mathcal{X}}$. We consider $D\pare{[0, T], {\mathcal{X}}}$ endowed with the Skorohod topology (see Chapter 3 of \cite{bil} for more details). Our process $\pare{Y_n\pare{t}, I_n\pare{t}}_{t\in [0, T]}$ is in $D\pare{[0, T], \overline{\Omega}}\times D\pare{[0, T],\brc{1,2}}$ which we consider endowed with the product topology.

\begin{theorem}\label{Teo.limit.process}
The sequence of processes converges in distribution 
\begin{align}\label{convd}
\pare{Y_n\pare{t}, I_n\pare{t}}\xrightarrow[n\to +\infty]{D}\pare{Y\pare{t}, I\pare{t}}
\end{align}
in $D\pare{[0, T], \overline{\Omega}}\times D\pare{[0, T], \brc{1,2}}$, where the distribution of the limit $\pare{Y\pare{t}, I\pare{t}}$ is characterized by 
having as probability densities $a(t,x)$ and $b(t,x)$ defined in \eqref{limite.debil.u.intro}, that is,
\begin{align}
P\Big( Y\pare{t}\in  E, I(t) =1 \Big) = \int_{E} a(t,z) \, dz \quad \mbox{and} \quad 
P \Big( Y\pare{t}\in  E, I(t) =2 \Big) = \int_{E} b(t,z) \, dz ,
\end{align}
for every measurable set $ E\subseteq \overline\Omega$.
\end{theorem}

In the following theorem we finally study the asymptotic behaviour of $u_n$, as $t\to \infty$, proving exponential convergence to
the unique stationary distribution.
\begin{theorem} \label{teo.comport.intro0} 
Let the initial condition be given by a distribution $u_0 \in L^2(\Omega)$.
There exist two constants $A>0$ (depending only on the domain and the kernels) and $C>0$
(that depends on the initial condition), such that
\begin{align}
\norm{u_n(t, \cdot) -\frac{1}{|\Omega|}}^2_{L^2(\Omega)}\leq Ce^{-At}.
\end{align}
\end{theorem}

Now, we fix a point $\bar{x}\in \Omega$ and analyze the case
in which the initial position is given by $Y_n(0) = \bar{x}$, that is, we assume that
\begin{align}
P\pare{Y_n\pare{0}\in E}= \delta_{\bar{x}}(E)=
\left\{
\begin{array}{ll}
1 \qquad & \bar{x} \in E, \\[5pt]
0 \qquad & \bar{x} \not\in E,
\end{array}
\right.
\end{align}
for every measurable set $E\subseteq \overline\Omega$.

In this case there is no convergence of the whole sequence $a_n$, $b_n$, but only convergence along
subsequences. This can be expected from the fact that the initial condition for $a_n$ (the initial condition fir $b_n$ is similar) 
satisfies
$$
\int_\Omega \varphi (x) a_n(0,x) dx = \int_\Omega \varphi(x) \chi_{A_n}(x) \delta_{\bar{x}}(dx) = 
\left\{
\begin{array}{ll}
\varphi (\bar{x}) \qquad & \bar{x} \in A_n, \\[5pt]
0 \qquad & \bar{x} \not\in A_n,
\end{array}
\right.
$$
that only converges along subsequences with $\bar{x} \in A_{n_j}$ or $\bar{x} \not\in A_{n_j}$ for every $n_j$.

Call $\pare{\nu^n_t}_{t}$ the law of $\pare{Y_n(t)}_t$.
By Dynkin's formula we know that for every $G$ continuous
\begin{align}
\begin{cases}
\displaystyle\frac{d}{dt}\int_\Omega G(x)\nu_t^n(dx)=\int_\Omega L_n G(x)\nu_t^n(dx),\\[6pt]
\displaystyle\int_\Omega G(x)\nu_0^n(dx)=G(\bar x),
\end{cases}
\end{align}
where $L_n$ is the generator of our process as described before. Since the involved kernels are smooth,
this evolution problem does not have a regularizing effect and therefore we expect that the measure $\delta_{\bar x}$,
that is the initial condition, remains for positive times. Hence, we write $\nu_t^n$ as an absolutely continuous part plus 
a time-dependent multiple of $\delta_{\bar x}$, that is, 
\begin{equation} \label{descomp.intro}
\nu_t^n(dx):=z_n(t, x)dx+\sigma_n(t)\delta_{\bar x}(dx).
\end{equation}
Now, we assume that
\begin{align}
&\chi_{A_n}(\bar x)\int_\Omega\chi_{A_n}(y) J(\bar x, y)dy+\chi_{B_n}(\bar x)\int_\Omega\chi_{B_n}(y) G(\bar x, y)dy
+\chi_{A_n}(\bar x)\int_\Omega\chi_{B_n}(y) R(\bar x, y)dy
+\chi_{B_n}(\bar x)\int_\Omega\chi_{A_n}(y) R(\bar x, y)dy=1.
\end{align}  
This condition says that the particle jumps with full probability (that is, the probability of staying at the same location when
the exponential clock rings is zero).
Under this condition, from the expression of $L_n$ we obtain that the time-dependent multiple of $\delta_{\bar{x}}$ is exponentially 
decreasing in time (independent on $n$) (see Section \ref{sect-deltas}) 
\begin{align}\label{due.intro}
\sigma_n(t)= e^{-t}, \quad \forall n\in\mathbb N.
\end{align}
This fact can be interpreted as follows: when the particle jumps for the first time the probability density passes
from being a delta at times $s<\tau_1$ to an absolutely continuous measure (recall that the kernels are smooth) for times 
greater $s>\tau_1$ and this first jump $\tau_1$ is distributed as an exponential of parameter 1. 

On the other hand, we have an equation for $z_n$,
\begin{align}\label{dom1.intro}
\frac{\partial z_n}{\partial t} (t, x)=&\chi_{A_n}(x)\int_\Omega\chi_{A_n}(y) J(x, y)\pare{z_n(t, y)-z_n(t, x)}dy+e^{-t} \chi_{A_n}(\bar x) \chi_{A_n}(x)J(\bar x, x)\\
&+\chi_{B_n}(x)\int_\Omega\chi_{B_n}(y) G(x, y)\pare{z_n(t, y)-z_n(t, x)}dy+ e^{-t}\chi_{B_n}(\bar x)\chi_{B_n}(x)G(\bar x, x)\\
&+\chi_{A_n}(x)\int_\Omega\chi_{B_n}(y) R(x, y)\pare{z_n(t, y)-z_n(t, x)}dy+ e^{-t}\chi_{A_n}(\bar x) \chi_{B_n}(x)R(\bar x, x)\\
&+\chi_{B_n}(x)\int_\Omega\chi_{A_n}(y) R(x, y)\pare{z_n(t, y)-z_n(t, x)}dy+ e^{-t} \chi_{B_n}(\bar x) \chi_{A_n}(x)R(\bar x, x),
\end{align}
with initial condition $z_n(0,x)=0$.

Notice that $z_n(t, x)$ is a function in $L^2(\Omega)$ for every $t>0$. Therefore, the solution to our evolution problem
with initial condition $\delta_{\bar{x}}$ is given by an absolutely continuos (with respect to the Lebesgue measure) part,
$z_n$, and a singular part, $e^{-t} \delta_{\bar{x}}$ (in this singular part the delta measure remains but decays exponentially fast in time). 

Now, we want to look at the limit as $n\to \infty$. Since the singular part of the solution, $e^{-t} \delta_{\bar{x}}$, is independent of $n$
we have to look for the behaviour of $z_n$ as $n\to \infty$ (here as we already mentioned we can only show convergence along subsequences).

 \begin{theorem}\label{teo.intro.med}
Given $(z_n)_{n\in \mathbb{N}}$ there is a subsequence $z_{n_k}$ that is weakly  convergent in $L^2((0,T) \times \Omega )$.
Moreover, it holds that
\begin{align}
\chi_{A_{\nk}}(x)z_{\nk}(t,x)\rightharpoonup a_k(t, x),\\
\chi_{B_{\nk}}(x)z_{\nk}(t,x)\rightharpoonup b_k(t, x),
\end{align}
weakly in $L^2((0,T) \times \Omega )$,
where $\pare{a_k(t, x),b_k(t, x)}$ is a solution to 
\begin{equation}\label{sys1.introa.22}
\left\{
\begin{array}{ll}
\displaystyle \frac{\partial a}{\partial t}\pare{t,x}=\int_{\Omega}J\pare{x, y}\pare{X(x)a\pare{t,y}-X(y)a\pare{t,x}}\, dy
\\[10pt]
\displaystyle \qquad \qquad \quad 
+\int_{\Omega}R\pare{x, y}\pare{X(x)b(t,y)-\pare{1-X(y)}a(t,x)}\, dy   +e^{-t}J(\bar x, x)X(x)\qquad &  \, t>0,\, x\in \Omega ,\\[10pt]
\displaystyle\frac{\partial b}{\partial t}\pare{ t,x}=\int_\Omega G\pare{x, y}\pare{\pare{1-X(x)}b\pare{t,y}-\pare{1-X(y)}b\pare{t,x}}dy\\[10pt]
\displaystyle\qquad\qquad+\int_\Omega R(x,y)\pare{\pare{1-X(x)}a(t,y)-X(y)b(t,x)}dy\, dy +e^{-t}R(\bar x, x)\pare{1-X(x)}
\qquad &  \, t>0,\, x\in \Omega ,\\[10pt]
a\pare{ 0,x}=0, \quad b\pare{0,x}=0
\qquad & x\in \overline\Omega,
\end{array} \right.
\end{equation}
or to
\begin{equation}\label{sys1.introb}
\left\{
\begin{array}{ll}
\displaystyle \frac{\partial a}{\partial t}\pare{t,x}=\int_{\Omega}J\pare{x, y}\pare{X(x)a\pare{t,y}-X(y)a\pare{t,x}}\, dy
\\[10pt]
\displaystyle \qquad \qquad \quad 
+\int_{\Omega}R\pare{x, y}\pare{X(x)b(t,y)-\pare{1-X(y)}a(t,x)}\, dy  +e^{-t}R(\bar x, x)X(x)\qquad & \, t>0,\, x\in \Omega , \\[10pt]
\displaystyle\frac{\partial b}{\partial t}\pare{t,x}=\int_\Omega G\pare{x, y}\pare{\pare{1-X(x)}b\pare{t,y}-\pare{1-X(y)}b\pare{t,x}}dy\\[10pt]
\displaystyle\qquad\qquad+\int_\Omega R(x,y)\pare{\pare{1-X(x)}a(t,y)-X(y)b(t,x)}dy\, dy
+e^{-t}G(\bar x, x)\pare{1-X(x)}
\qquad &  \, t>0,\, x\in \Omega ,\\[10pt]
a\pare{0,x}=0, \quad b\pare{0,x}=0
\qquad & x\in \overline\Omega.
\end{array} \right.
\end{equation}
\end{theorem}

The first system, $\eqref{sys1.introa.22}$, occurs when the convergent subsequence is such that $\bar{x} \in A_{n_k}$
for every $n_k$; while the second one, \eqref{sys1.introb}, appears when $\bar{x} \in B_{n_k}$
for every $n_k$.

Notice that the two possible limit systems are similar but different since in \eqref{sys1.introa.22} we have exponential terms
like $e^{-t}J(\bar x, x)X(x)$ and $e^{-t}R(\bar x, x)\pare{1-X(x)}$
 while in
\eqref{sys1.introb} the terms $e^{-t}R(\bar x, x)X(x)$ and $e^{-t}G(\bar x, x)\pare{1-X(x)}$ appear. 
Also remark that both systems \eqref{sys1.introa.22} and  \eqref{sys1.introb} are similar to \eqref{sys1.intro}
except by the fact that the exponential terms do not appear in \eqref{sys1.intro} (c.f. Theorem \ref{teo.main.intro}). 
In addition, we have that also the limit $u(t,x)=a(t,x)+b(t,x)$ is different in the two previously mentioned cases 
(notice that in the first system the term $e^{-t}G(\bar x, x)\pare{1-X(x)}$ that involves the kernel $G$ does not appear;
while in the second one the term $e^{-t}J(\bar x, x)X(x)$ is missing).  

We finally analyze the asymptotic behaviour of $z_n$, as $t\to \infty$, proving exponential convergence to
the unique stationary distribution.
\begin{theorem} \label{teo.comport.intro} There exist $A>0$ and $C>0$, such that
\begin{align}
\norm{z_n(t, \cdot) -\frac{1}{|\Omega|}}^2_{L^2(\Omega)}\leq Ce^{-At},
\end{align}
for $t$ large enough.
\end{theorem}
As a consequence of this theorem we obtain the asymptotic behaviour for $\pare{\nu^n_t}_{t}$ the law of our process
$\pare{Y_n(t)}_t$ starting at $\delta_{\bar{x}}$. For every continuous function $g$, there exist $\widehat{A}>0$  
and $\widehat{C}>0$ (independent of $g$ and $\bar{x}$) such that
$$
\left| \int_\Omega {g}(x)\nu_t^n(dx) - \frac{1}{|\Omega|} \int_\Omega {g}(x)\, dx  \right| =
\left| \int_\Omega {g}(x)z_n(t,x) \, dx + e^{-t } {g}(\bar{x}) - \frac{1}{|\Omega|} \int_\Omega {g}(x)\, dx   \right| \leq \| {g}\|_{L^\infty (\Omega)} 
\widehat{C}e^{-\widehat{A}t}.
$$
That is, we have that $\pare{\nu^n_t}$ converges exponentially in the sense of measures
to the unique stationary distribution as $t \to +\infty$.

Now, let us end the introduction with a brief description of previous results and comments on the ideas and difficulties involved in our proofs. 

Nonlocal equations with smooth kernels like the ones considered here has been widely studied and used in the literature 
as models in different applied scenarios, see for example, \cite{ElLibro,BCh,CF,Chasseigne-Chaves-Rossi-2006,delia1,delia2,delia3,F,Gal}.
Here we have a model in which the jumping probabilities depend on three different kernels $J$, $G$ and $R$
that act in different parts of the domain (thus, our model problem can be seen as a coupling between two nonlocal equations that
occur in the sets $A$ and $B$).
For other couplings (even considering local equations and nonlocal ones) we refer to \cite{CaRo,delia1,delia2,delia3,Du,GQR,Gal,Kra}.

Homogenization for PDEs is by now a classical subject that originated in the study of the behaviour
of the solutions to elliptic and parabolic local equations with highly oscillatory coefficients (periodic homogenization). 
We refer to
\cite{Ben,TT, Tar} as general references for the subject.
For other kinds of homogenization for pure nonlocal problems with one kernel we
refer to \cite{marc1,marc2,marc3}. For homogenization results for nonlocal equations with a singular 
kernel (like the one that appears in the fractional Laplacian) we refer to \cite{Ca,sw2,Wa} and references therein. We emphasize that the previously mentioned references
deal with homogenization in the coefficients involved in the equation. For random homogenization of an obstacle problem we refer to \cite{caffa2}. 
For mixing local and nonlocal processes we refer to \cite{CaRo}.
Here we deal with an homogenization problem that is different in nature with the ones treated in the previously mentioned references
as we homogenize mixing three different jump operators with smooth kernels.

Finally, let us describe the main ingredients that appear in the proofs. First, we show 
weak convergence along subsequences of $u_n$, $\chi_{A_n}u_n$ and $\chi_{B_n}u_n$ (these convergences comes from
a uniform bound in $L^2$). Next, we find the system that these limits verify. This part of the proof
is delicate since we have to  
pass to the limit in the week form of the equation that involves terms like $\chi_{B_n} (x) \chi_{A_n} (y) {u_n (t,y)} J(x-y)$ and we only have weak convergence of $\chi_{B_n}$ and
$\chi_{A_n} u_n$. 
Here we need to rely on the continuity of $J$ and use the fact that the product
$\chi_{B_n} (x) \chi_{A_n} (y) {u_n (t,y)} J(x-y)$ involves two different variables, $x$ and $y$. Finally, we show
uniqueness of the limit by proving uniqueness of solutions to the limit system.

To show the convergence of the process we first prove that $Y_n(t)$ has a probability density $u_n( t,x)$ which is the unique solution to system \eqref{evol.intro}. Next, we prove that the laws of the processes $\pare{Y_n\pare{t}, I_n\pare{t}}_{t\in [0, T]}$ form a tight sequence and finally we characterize the limit as the limit process.

	When the initial condition is $u_0=\delta_{\bar{x}}$ we use analogous arguments, but in this case we
	need extra care since, due to the lack of regularizing effect, we have a term of the form $e^{-t} \delta_{\bar{x}}$
	in the solution (see formula \eqref{descomp.intro}). This creates the extra exponential terms in the equation
	satisfied by $z_n$, \eqref{dom1.intro}. We remark again that here there is only convergence along subsequences for which
	the point at which the process starts, $\bar{x}$, satisfies that $\bar{x}\in A_{n_k}$ or $\bar{x}\in A_{n_k}$ for every $n_k$.
	Notice that these two possible limits along subsequences are different (the limits are solutions to two different systems), so
	in general, the full limit does not exists.

\medskip
	
The paper is organized as follows: in Section \ref{sect-u0L2} we analyze the case in which  $u_0 \in L^2(\Omega)$  by proving convergence of the densities and convergence of the processes as $n\to\infty$, and also analyzing the asymptotic behaviour of the densities as $t\to\infty$; in Section \ref{sect-deltas}
	we discuss the convergence via subsequences when $u_0 = \delta_{\bar x}$ and the asymptotic behaviour of $z_n$. 
	Finally in Section \ref{sect-Dirichlet} we include a brief description of the same problem with Dirichlet boundary conditions.

\section{Initial conditions $u_0 \in L^2(\Omega)$.} \label{sect-u0L2}

\subsection{Convergence of the densities.}  \label{sect-densities}
This subsection is dedicated to the proof of Theorem \ref{teo.main.intro}.
We start by showing the following lemma which guarantees that the sequence $u_n$ is uniformly bounded in the $L^2$ norm.

\begin{lemma} \label{lema-3-2-}
Let $u_n$ be the solution of \eqref{evol.intro}. Then there exists a constant $C$ such that
\begin{align}
\norm{u_n}_{L^\infty \pare{0, T: L^2(\Omega)}}\leq C.
\end{align}
\end{lemma}

\begin{proof}
To prove the uniform bound we just multiply by $u_n$ both sides of \eqref{evol.intro} and integrate in $\Omega$ and in $[0,T]$ to obtain
\begin{align}\label{wq}
\begin{array}{l}
\displaystyle 
\frac12\int_\Omega (u_n)^2 (T,x) dx -  
\frac12\int_\Omega (u_0)^2 (x) dx = \int_0^T \int_\Omega   {L}_n u_n(t,x) u_n (t,x) dx dt  \\[10pt]
\qquad \displaystyle 
= \int_0^T \int_\Omega \chi_{A_n}\pare{x}\int_\Omega \chi_{A_n}\pare{y}  J\pare{x,y}\pare{u_n\pare{t,y}-u_n\pare{t,x}}dy u_n (t,x) 
dx dt 
\\[10pt]
\qquad \displaystyle 
\qquad + \int_0^T \int_\Omega \chi_{A_n}\pare{x}\int_\Omega \chi_{B_n}\pare{y}R\pare{x,y}\pare{u_n\pare{t,y}- u_n\pare{t,x}}dy
u_n (t,x) 
dx dt 
\\[10pt]
\qquad \displaystyle 
\qquad + \int_0^T \int_\Omega \chi_{B_n}\pare{x}\int_\Omega \chi_{B_n}\pare{y}G\pare{x,y}\pare{u_n\pare{t,y}- u_n\pare{t,x}}dy
u_n (t,x) 
dx dt 
\\[10pt]
\qquad \displaystyle 
\qquad + \int_0^T \int_\Omega \chi_{B_n}\pare{x}\int_\Omega \chi_{A_n}\pare{y}R\pare{x,y}\pare{u_n\pare{t,y}- u_n\pare{t,x}}dy
u_n (t,x) 
dx dt
\\[10pt]
\qquad \displaystyle 
\leq  - \int_0^T \int_{\Omega \times \Omega} (1- \chi_{B_n}\pare{x}
\chi_{B_n}\pare{y} ) \pare{J(x,y)+G(x, y)+2R(x, y)} \pare{u_n\pare{t,y}-u_n\pare{t,x}}^2 dx dy dt \leq 0,
\end{array}
\end{align}
and hence the $L^2-$norm of the solution is decreasing in time and the result follows. 
\end{proof}

Now, consider $$a_n(t,x):=\chi_{A_n} (x) u_n(t,x) \qquad \mbox{ and } \qquad b_n(t,x):=\chi_{B_n} (x) u_n(t,x).$$
We are ready to proceed with the proof of Theorem \ref{teo.main.intro}.

\begin{proof}[Proof of Theorem \ref{teo.main.intro}] By Lemma \eqref{lema-3-2-} we can extract a weakly convergent subsequence of $a_n( t,x)$ and $b_n( t,x)$ that for simplicity of notation we index again with $n$. We call $a$ and $b$ their respective weak limits.

Take a smooth function $\phi$ such that $\phi(T,\cdot)\equiv 0$ and consider equation \eqref{evol.intro}. 
Multiply both sides by $\chi_{B_n}\pare{x}\phi\pare{ t,x}$ and then integrate respect to the variables $x$
and $t$. Since by construction $\phi( T,\cdot)\equiv 0$, integrating by parts we obtain
\begin{align}
-\int_0^T\int_\Omega &\frac{\partial \phi}{\partial t}( t,x)b_n\pare{ t,x}\, dxdt-\int_\Omega \chi_{B_n}\pare{x}u_0\pare{x}\phi\pare{0,x}\, dx\\[10pt]
&\hspace{-25pt}=\int_0^T\int_\Omega\int_{\Omega}\chi_{B_n}\pare{x}\chi_{B_n}\pare{y} G\pare{x,y}\pare{u_n\pare{t,y}-u_n\pare{t,x}}\phi\pare{t,x}dydxdt\\
&+\int_0^T\int_\Omega\int_{\Omega}\chi_{B_n}\pare{x}\chi_{A_n}\pare{y}R\pare{x,y}\pare{u_n\pare{t,y}-u_n\pare{t,x}}\phi\pare{t,x}dydxdt\\
&\hspace{-25pt}=\int_0^T\int_\Omega \int_\Omega G\pare{x, y}\chi_{B_n}\pare{x}b_n\pare{t,y}\phi\pare{ t,x}\, dydxdt-\int_0^T\int_\Omega \int_\Omega G\pare{x, y}\chi_{B_n}\pare{y}b_n\pare{ t,x}\phi\pare{ t,x}\, dydxdt\\[10pt]
&+\int_0^T\int_\Omega \int_\Omega R\pare{x, y}\chi_{B_n}\pare{x}a_n\pare{ t,y}\phi\pare{ t,x}\, dydxdt-\int_0^T\int_\Omega \int_\Omega R\pare{x, y}\chi_{A_n}\pare{y}b_n\pare{ t,x}\phi\pare{ t,x}\, dydxdt.
\end{align}
Since $$\chi_{A_n}\pare{\cdot}\xrightarrow[n\to +\infty]{}X\pare{\cdot} \qquad \mbox{ and } \qquad \chi_{B_n}\pare{\cdot}\xrightarrow[n\to +\infty]{}1-X\pare{\cdot}$$ weakly in $L^2\pare{\Omega \times (0,T)}$, we obtain the following limits 
\begin{align}\label{formula1}
\int_0^T\int_\Omega \frac{\partial \phi}{\partial t} \pare{ t,x}b_n\pare{ t,x}\, dxdt
\xrightarrow[n\to +\infty]{}\int_0^T\int_\Omega \frac{\partial \phi}{\partial t} \pare{ t,x}b\pare{t,x}\, dxdt,
\end{align}
\begin{align}\label{formula2}
\int_\Omega \chi_{B_n}\pare{x}u_0\pare{x}\phi\pare{0,x}\, dx\xrightarrow[n\to +\infty]{}\int_\Omega\pare{1-X\pare{x}}u_0(x)\phi\pare{0,x}\, dx.
\end{align}

Now, as we assumed that $G(x,y)$ is continuous, we have that
\begin{align}\label{formula3}
h_n(y) = \int_\Omega  G(x,y) \chi_{B_n} (x) \phi\pare{ t,x} \, dx   
\xrightarrow[n\to +\infty]{} \int_\Omega  G(x,y) \pare{1-X(x)} \phi \pare{t,x} \, dx 
\end{align}
uniformly in $y$. Therefore, we get
\begin{align}\label{formula4}
\int_0^T\int_\Omega \int_\Omega G\pare{x, y}\chi_{B_n}\pare{x}b_n\pare{t,y}\phi\pare{t,x}\, dydxdt\xrightarrow[n\to +\infty]{}\int_0^T\int_\Omega 
\int_\Omega G\pare{x, y}\pare{1-X\pare{x}}b\pare{t,y}\phi(t,x)\, dydxdt,
\end{align}
and, arguing similarly,
\begin{align}\label{formula5}
\begin{aligned}
\displaystyle
\int_0^T\int_\Omega \int_\Omega G\pare{x, y}\chi_{B_n}\pare{y}b_n\pare{t,x}\phi\pare{t,x}\, dydxdt
\xrightarrow[n\to +\infty]{}
\int_0^T\int_\Omega  \int_\Omega G\pare{x, y}\pare{1-X\pare{y}}b\pare{t,x}\phi\pare{t,x}\, dydxdt,\\
\int_0^T\int_\Omega \int_\Omega R\pare{x, y}\chi_{B_n}\pare{x}a_n\pare{t,y}\phi\pare{t,x}\, dydxdt
\xrightarrow[n\to +\infty]{}
\int_0^T\int_\Omega  \int_\Omega R\pare{x, y}\pare{1-X\pare{x}}a\pare{t,y}\phi\pare{t,x}\, dydxdt,\\
\int_0^T\int_\Omega \int_\Omega R\pare{x, y}\chi_{B_n}\pare{y}b_n\pare{t,x}\phi\pare{t,x}\, dydxdt,
\xrightarrow[n\to +\infty]{}
\int_0^T\int_\Omega  \int_\Omega R\pare{x, y}X\pare{y}b\pare{t,x}\phi\pare{t,x}\, dydxdt.
\end{aligned}
\end{align}

Collecting all these limits we conclude that
\begin{align}
-\int_0^T\int_\Omega & \frac{\partial{\phi}}{\partial t} \pare{t,x}b \pare{t,x}dxdt-\int_\Omega \pare{1-X \pare{x}}u_0\pare{x}\phi \pare{0,x}\, dx\\[10pt]
&=\int_0^T\int_\Omega  \int_\Omega G\pare{x, y}\pare{1-X\pare{x}}b\pare{t,y}\phi\pare{t,x}\, dydxdt-\int_0^T\int_\Omega  \int_\Omega G\pare{x, y}\pare{1-X\pare{y}}b\pare{t,x}\phi\pare{t,x}\, dydxdt\\
&\qquad +\int_0^T\int_\Omega  \int_\Omega R\pare{x, y}\pare{1-X\pare{x}}a\pare{t,y}\phi\pare{t,x}\, dydxdt-\int_0^T\int_\Omega  \int_\Omega R\pare{x, y}X\pare{y}b\pare{t,x}\phi\pare{t,x}\, dydxdt.
\end{align}
Since this holds for every $\phi$, we conclude that $b\pare{t,x}$ is a solution to
\begin{align}
\begin{cases}
& \displaystyle
\frac{\partial b}{\partial t}\pare{t,x}=\int_\Omega G\pare{x, y}\cor{\pare{1-X(x)}b\pare{t,y}-\pare{1-X(y)}b\pare{t,x}}dy+\int_\Omega 
R(x,y)\cor{\pare{1-X(x)}a(t,y)-X(y)b(t,x)}dy\\[10pt]
&b\pare{x, 0}=\pare{1-X\pare{x}}u_0\pare{x}.
\end{cases}
\end{align}
In a similar way we get the equation for $a(x,t)$ and this concludes the proof of Theorem \ref{teo.main.intro}.
\end{proof}

{
Next, let us prove a corrector result. We proceed as in \cite{nosotros, marc3} setting the corrector as 
$$
\omega_n(t,x) = \dfrac{\chi_{A_n}(x) a(t,x)}{X(x)} + \dfrac{\chi_{B_n}(x) b(t,x)}{1-X(x)} \quad (t,x) \in [0,T] \times \Omega.
$$

\begin{corollary}
Under the conditions of Theorem \ref{teo.main.intro}, we have for each $t \in [0,T]$ that  
\begin{equation} \label{n630}
	\left\|u_n(t,\cdot) - \omega_n(t,\cdot) \right\|_{L^2(\Omega)} \to 0 \quad \textrm{ as } n \to +\infty.
\end{equation}
\end{corollary}
\begin{proof}
First, we observe that we can use the variation of constants formula to write the solutions of \eqref{evol.intro} as 
\begin{eqnarray} \label{solfvc}
u_n(t,x) & = & e^{-m_n(x) t} u_0(x) + \int_0^t e^{-m_n(x)(t-s)} \int_\Omega H_n(x,y) u_n(s,y) dy ds
\end{eqnarray}
where 
$$
\begin{gathered}
m_n(x) = \chi_{A_n}(x) \int_\Omega \left( \chi_{A_n}(y) J(x,y) + \chi_{B_n}(y) R(x,y) \right) dy 
+ \chi_{B_n}(x) \int_\Omega \left( \chi_{B_n}(y) G(x,y) + \chi_{A_n}(y) R(x,y) \right) dy  \\
\textrm{ and } \quad 
H_n(x,y) = \chi_{A_n}(x) \left( \chi_{A_n}(y) J(x,y) + \chi_{B_n}(y) R(x,y) \right) + \chi_{B_n}(x) \left( \chi_{B_n}(y) G(x,y) + \chi_{A_n}(y) R(x,y) \right).
\end{gathered}
$$
Also, we have that
\begin{equation} \label{ncor}
\left\|u_n(t,\cdot) - \omega_n(t,\cdot) \right\|_{L^2(\Omega)}^2 = \| u_n(t,\cdot) \|_{L^2(\Omega)}^2 - 2 \int_\Omega u_n(t,x) \, \omega_n(t,x) \, dx 
+ \| \omega_n(t,\cdot) \|_{L^2(\Omega)}^2.
\end{equation}
We will obtain the result passing to the limit in each term of \eqref{ncor}.
Now, in order to do that, we need to study the sequence of functions $\{ e^{-m_n(\cdot) t} \}_{n\in \N} \subset L^\infty(\Omega)$ and $\{ \int_\Omega H_n(\cdot,y) u_n(t,y) \, dy \}_{n \in \N} \subset L^\infty(\Omega)$ as $n \to + \infty$. 

First, let us consider $\{ e^{-m_n(x) t} \}_{n\in \N}$. For all $\varphi \in L^1(\Omega)$ one has 
\begin{eqnarray*}
\int_\Omega \varphi(x) e^{-m_n(x) t} dx & = & \displaystyle \int_\Omega \varphi(x) e^{- t \Big( \chi_{A_n}(x) \int_\Omega \left\{ \chi_{A_n}(y) J(x,y) + \chi_{B_n}(y) R(x,y) \right\} dy \Big)} e^{-t \Big(\chi_{B_n}(x) \int_\Omega \left\{ \chi_{B_n}(y) G(x,y) + \chi_{A_n}(y) R(x,y) \right\} dy \Big)}dx \\
 & = &  \displaystyle \int_\Omega \chi_{A_n}(x) \, \varphi(x) \, e^{- t \Big( \int_\Omega \left\{ \chi_{A_n}(y) J(x,y) + \chi_{B_n}(y) R(x,y) \right\} dy \Big)} dx \\
& &  \displaystyle \quad + \int_\Omega \chi_{B_n}(x) \, \varphi(x) \, e^{-t \Big(\int_\Omega \left\{ \chi_{B_n}(y) G(x,y) + \chi_{A_n}(y) R(x,y) \right\} dy \Big)}dx \\
 & = &  \displaystyle \int_\Omega \chi_{A_n}(x) \, \varphi(x) \, e^{- t m_n^1(x)} dx + \int_\Omega \chi_{B_n}(x) \, \varphi(x) \, e^{-t m_n^2(x)}dx
\end{eqnarray*}
where
$$
\begin{gathered}
m_n^1(x) = \int_\Omega \left\{ \chi_{A_n}(y) J(x,y) + \chi_{B_n}(y) R(x,y) \right\} dy \quad \textrm{ and } \quad
m_n^2(x) = \int_\Omega \left\{ \chi_{B_n}(y) G(x,y) + \chi_{A_n}(y) R(x,y) \right\} dy. 
\end{gathered}
$$

Notice that, for each $x \in \Omega$, we have
\begin{equation} \label{eq770}
\begin{gathered}
m_n^1(x)
\to m^1(x) := \int_\Omega \left\{  X(y) J(x,y) + (1- X(y)) R(x,y) \right\} dy  \\
\quad \textrm{ and } \quad \\
m_n^2(x)
\to  m^2(x) := \int_\Omega \left\{ (1 - X(y)) G(x,y) + X(y) R(x,y) \right\} dy
\end{gathered}
\end{equation}
as $n \to \infty$. 
Hence, from \cite[Proposition 2.1]{marc3} we obtain that 
\begin{equation} \label{eq780}
\begin{gathered}
e^{- t m_n^1(x)}  
\to e^{ -t m^1(x)} \quad \textrm{ and } \quad 
e^{-t m^2_n(x))} 
\to e^{-t m^2(x)}
\end{gathered}
\end{equation}
strongly in $L^\infty(\Omega)$ for each $t \in [0,T]$.

In particular, for all $\varphi \in L^1(\Omega)$, one has 
\begin{eqnarray*}
\int_\Omega \varphi(x) e^{-m_n(x) t} dx & \to & \int_\Omega X(x) \, \varphi(x) \, e^{ -t \Big( \int_\Omega \left\{  X(y) J(x,y) + (1- X(y)) R(x,y) \right\} dy \Big) } dx \\
& & \quad + \int_\Omega (1-X(x)) \, \varphi(x) \, e^{-t \Big( \int_\Omega \left\{ (1 - X(y)) G(x,y) + X(y) R(x,y) \right\} dy \Big)} dx.
\end{eqnarray*}

On the other hand, we can write 
$$
\int_\Omega H_n(x,y) u_n(t,y) \, dy  = \chi_{A_n}(x) \Phi_n^1(t,x) + \chi_{B_n}(x) \Phi_n^2(t,x)
$$
where 
\begin{equation} \label{eq810}
\begin{gathered} 
\Phi_n^1(t,x) =  \int_\Omega \left\{ \chi_{A_n}(y) J(x,y) + \chi_{B_n}(y) R(x,y) \right\}  u_n(t,y) \, dy \\ \quad \textrm{ and } \\
\Phi_n^2(t,x) =  \int_\Omega \left\{ \chi_{B_n}(y) G(x,y) + \chi_{A_n}(y) R(x,y) \right\}  u_n(t,y) \, dy.
\end{gathered}
\end{equation}
We can argue as in \eqref{eq780} to obtain 
\begin{equation} \label{eq810.99}
\begin{gathered}
\Phi_n^1(t,x)
\to \Phi^1(t,x) := \int_\Omega \left\{ a(t,y) \, J(x,y) + b(t,y) \, R(x,y) \right\} dy \\ 
\quad 
\textrm{ and } \\
\Phi_n^2(t,x)
\to \Phi^2(t,x) := \int_\Omega \left\{ b(t,y) \, G(x,y) + a(t,y) \, R(x,y) \right\} dy
\end{gathered}
\end{equation}
strongly in $L^\infty(\Omega)$ for each $t \in [0,T]$.

Now, let us pass to the limit in $\| u_n(t,\cdot) \|_{L^2(\Omega)}$. Due to \eqref{solfvc}, we get from \eqref{eq770} and \eqref{eq810.99} that 
\begin{eqnarray*}
&&  \int_{\Omega} u_n^2(t,x) \,dx  = 
\displaystyle
\int_{\Omega} u_n(t,x) \left( e^{-m_n(x) t} \, u_0(x) 
+   \int_{0}^t e^{-m_n(x) (t-s)} \int_{\Omega} H_n(x,y) \, u_n(s,y) \, dy ds \right) dx  \\
& &  \qquad = 
\displaystyle
\int_\Omega \chi_{A_n}(x) \, u_n(t,x) \, e^{- t m_n^1(x)} u_0(x) \, dx + \int_\Omega \chi_{B_n}(x) \, u_n(t,x) \, e^{-t m_n^2(x)} u_0(x) \, dx \\
& & \qquad \qquad + \displaystyle
\int_0^t  \int_\Omega \chi_{A_n}(x) \, u_n(t,x) \, e^{- t m_n^1(x)} 
\Phi_n^1(s,x) \, dx ds  + \int_0^t \int_\Omega \chi_{B_n}(x) \, u_n(t,x) \, e^{-t m_n^2(x)} 
\Phi_n^2(s,x) \, dx ds .
\end{eqnarray*}
Consequently, it follows from Theorem \ref{teo.main.intro}, \eqref{eq780} and \eqref{eq810.99} that 
\begin{equation} \label{eq850}
\begin{array}{l}
\| u_n(t,\cdot) \|_{L^2(\Omega)}^2  \to   
\displaystyle
\int_\Omega a(t,x) \, e^{- t m^1(x)} u_0(x) \, dx + \int_\Omega b(t,x) \, e^{-t m^2(x)} u_0(x) \, dx \\[10pt]
 \qquad \qquad + \displaystyle
\int_0^t  \int_\Omega a(t,x) \, e^{- t m^1(x)} \Phi^1(s,x) \, dx ds 
+ \int_0^t \int_\Omega b(t,x) \, e^{-t m_n^2(x)} \Phi^2(s,x) \, dx ds \\[10pt]
 \qquad = \displaystyle \int_\Omega \frac{a(t,x)}{X(x)} \left[  (X u_0)(x) e^{-m^1(x) t} 
 \displaystyle + \int_0^t e^{-m^1(x) (t-s)} \left\{ X(x) J(x,y) + (1-X(x)) R(x,y) \right\} dy ds \right] dx 
  \\[10pt]  \displaystyle \qquad  \qquad
 +\int_\Omega \frac{b(t,x)}{1-X(x)} \left[  ((1-X) u_0)(x) e^{-m^2(x) t} 
 + \int_0^t e^{-m^2(x) (t-s)} \left\{ (1-X)(x) G(x,y) + X(x) R(x,y) \right\} dy ds \right] dx  \\[10pt]
 \qquad  =  \displaystyle \int_\Omega \left\{ \frac{a^2(t,x)}{X(x)} +  \frac{b^2(t,x)}{1-X(x)}  \right\} dx,
\end{array}
\end{equation}
since we have that 
$$
\begin{gathered}
a(t,x) = (X u_0)(x) e^{-m^1(x) t} + \int_0^t e^{-m^1(x) (t-s)} \left\{ X(x) J(x,y) + (1-X(x)) R(x,y) \right\} dy ds \\
\quad \textrm{ and } \quad \\
b(t,x) = ((1-X) u_0)(x) e^{-m^2(x) t} + \int_0^t e^{-m^2(x) (t-s)} \left\{ (1-X)(x) G(x,y) + X(x) R(x,y) \right\} dy ds.
\end{gathered}
$$

Finally, let us pass to the limit in the other terms of \eqref{ncor}. One can see that 
\begin{eqnarray*} \label{eq870}
\int_\Omega u_n(t,x) \, \omega_n(t,x) \, dx 
 =  \int_\Omega \left\{ \chi_{A_n}(x) u_n(t,x) \frac{a(t,x)}{X(x)} + \chi_{B_n}(x) u_n(t,x) \frac{b(t,x)}{1-X(x)} \right\} dx  \to  \int_\Omega \left\{  \frac{a^2(t,x)}{X(x)} + \frac{b^2(t,x)}{1-X(x)} \right\} dx 
\end{eqnarray*}
and 
\begin{eqnarray*} \label{eq880}
\int_\Omega \omega^2_n(t,x) \, dx 
 = \int_\Omega \left\{ \chi_{A_n}(x) \frac{a^2(t,x)}{X^2(x)} + \chi_{B_n}(x) \frac{b^2(t,x)}{(1-X(x))^2} \right\} dx 
 \to  \int_\Omega \left\{  \frac{a^2(t,x)}{X(x)} + \frac{b^2(t,x)}{1-X(x)} \right\} dx, \qquad \textrm{ as } n \to +\infty.
\end{eqnarray*}
Hence, we can conclude that 
$$
\lim_{n \to +\infty} \left\|u_n(t,\cdot) - \omega_n(t,\cdot) \right\|_{L^2(\Omega)}^2 = 0
$$
proving the result.
\end{proof}

}

\subsection{Convergence of the stochastic process.} \label{Sect.proc}
In this subsection we  prove Theorem \ref{Teo.limit.process}.
Our first goal is to show that $Y_n(t)$ has a probability density $u_n( t,x)$ which is the unique solution to system \eqref{evol.intro}.
To this end we will prove uniqueness of weak solutions to \eqref{evol.intro}.

\begin{lemma}\label{lemma:1}
Let $u_0\in L^2\pare{\Omega}$. There exists a unique solution $u_n\pare{x,t}\in L^2\pare{\Omega\times [0, T]}$ of system 
\begin{equation}\label{evol.intro12}
\left\{
\begin{array}{ll}
\displaystyle \frac{\partial u_n}{\partial t} (t,x) = {L}_n u_n(t,x), \qquad & \, t>0, \, x\in \Omega, \\[10pt]
u_n(0,x)=u_0(x), \qquad & x\in \Omega.
\end{array}
\right.
\end{equation}
\end{lemma}

\begin{proof}
Fix $t_0\in [0, T]$ and consider the space $L^2\pare{\Omega\times [0, t_0]}$ endowed with the norm $\norm{\cdot}_{2, \infty}$ defined as
\begin{align}
\norm{f}_{2,\infty}=\sup_{t\in [0, t_0]}\int_{\Omega}  [f(t,x)]^2dx
\end{align}
and $L^\infty\pare{\Omega\times [0, t_0]}$ with the norm
$$\|f\|_{\infty,\infty}:=\sup\Big\{ \abs{f(s,x)}, \, x\in\Omega, s\in[0,t]  \Big\}.$$
We let $\Phi$ the map defined on $u\in L^2\pare{\Omega\times [0, t_0]}$ as
\begin{align}\label{wast}
\begin{aligned}
\Phi\pare{u(t,x)}=&u_0(x) +  \int_0^t\chi_{A_n}\pare{x}\cor{\int_{\Omega}\chi_{A_n}\pare{y} J\pare{x,y}\pare{u\pare{s,y}-u\pare{s,x}}dy+\int_\Omega \chi_{B_n}\pare{y}R\pare{x,y}\pare{u\pare{s,y}-u\pare{s,x}}dy}ds\\
&+\int_0^t\chi_{B_n}\pare{x}\cor{\int_{\Omega}\chi_{B_n}\pare{y} G\pare{x,y}\pare{u\pare{s,y}-u\pare{s,x}}dy+\int_\Omega \chi_{A_n}\pare{y}R\pare{x,y}\pare{u\pare{s,y}-u\pare{s,x}}dy}ds.
\end{aligned}
\end{align}
For every $u, v\in L^2\pare{\Omega\times [0, t_0]}$ it holds that
\begin{align}\label{eqr1}
\abs{\Phi\pare{u(t,x)}-\Phi\pare{v(t,x)}}&\leq\Big(\norm{J}_\infty+2\norm{R}_\infty+\norm{G}_\infty  \Big)\int_0^t\int_{\Omega}\pare{\abs{u(s,y)-v(s,y)}+\abs{u(s,x)-v(s,x)}}dyds.
\end{align}
Calling $\sqrt{C}=\norm{J}_\infty+2\norm{R}_\infty+\norm{G}_\infty$ we get that
\begin{align}
\begin{aligned}
 \int_\Omega \abs{\Phi\pare{u(t,x)}-\Phi\pare{v(t,x)}}^2 dx 
&\leq C\int_{\Omega}\pare{\int_0^t\int_{\Omega}\pare{\abs{u(s,y)-v(s,y)}+\abs{u(s,x)-v(s,x)}}dyds}^2dx\\
&\leq 2C\int_{\Omega}\pare{\int_0^t\int_{\Omega}\abs{u(s,y)-v(s,y)}dyds}^2dx
+2C\int_{\Omega}\pare{\int_0^t\int_{\Omega}\abs{u(s,x)-v(s,x)}dyds}^2dx.
\end{aligned}
\end{align}
By Cauchy-Schwartz's inequality we obtain
\begin{align*}
\pare{\int_0^t\int_{\Omega}\abs{u(s,y)-v(s,y)}dyds}^2&\leq t\int_0^t\pare{\int_{\Omega}\abs{u(s,y)-v(s,y)}dy}^2ds\\
&\leq t\abs{\Omega}\int_0^t\int_{\Omega}\abs{u(s,y)-v(s,y)}^2dyds\\
&\leq t^2\abs{\Omega}\norm{u-v}_{2, \infty},
\end{align*}
and, analogously,
\begin{align*}
\pare{\int_0^t\int_{\Omega}\abs{u(s,x)-v(s,x)}dyds}^2&\leq t\abs{\Omega}^2\int_0^t\abs{u(s,x)-v(s,x)}^2ds.
\end{align*}
Consequently, by \eqref{eqr1}, we get
\begin{align}
\begin{aligned}
 \int_\Omega \abs{\Phi\pare{u(t,x)}-\Phi\pare{v(t,x)}}^2 dx  
&\leq 2Ct^2\abs{\Omega}^2\norm{u-v}_{2, \infty}+2Ct\abs{\Omega}^2\int_{\Omega}\int_0^t
\abs{u(s,x)-v(s,x)}^2dsdx\\
&\leq 4Ct^2\abs{\Omega}^2\norm{u-v}_{2, \infty}.
\end{aligned}
\end{align}
Therefore by choosing $t_0<\frac{1}{2\abs{\Omega}\sqrt{C}}$ we get that the map $\Phi$ is a contraction in $L^2\pare{ [0, t_0] \times \Omega}$.
By the Banach fixed-point Theorem we can deduce that there exists a unique solution of system \eqref{evol.intro12} in 
$L^2\pare{ [0, t_0] \times \Omega}$. 
We can iterate the previous argument in order to show existence and uniqueness globally in $L^2\pare{[0, T] \times \Omega}$.
\end{proof}

\begin{lemma}\label{lemma:uniqweak}
There exists a unique measure $\nu_t$ solution to 
\begin{align}\label{weak:eq2}
\begin{cases}
\displaystyle \frac{\partial}{\partial t}\int_{\overline\Omega}f(x)\nu_t(dx)=\int_{\overline\Omega} {{L}_n} f(x)\nu_t(dx),&\quad 
 \; t\in [0, T], \, x\in\overline\Omega,
\\[10pt]
\displaystyle\int_{\overline\Omega}f(x)\nu_0(dx)=\int_{\overline\Omega}f(x)u_0(x)dx, &\quad x\in\overline\Omega,
\end{cases}
\end{align}
for every $f\in C\pare{A_n}\cap C(B_n)$.

Such solution is given by $$\nu^n_t(dx)=u_n(t,x)dx,$$ where
$u_n(t,x)$ is the unique solution to \eqref{evol.intro}.
\end{lemma}

\begin{proof}
The existence of a solution to \eqref{weak:eq2} follows just by taking 
$$\nu_t^n(dx)=u_n(t,x)dx$$ 
where $u_n(t,x)$ is a solution of system \eqref{evol.intro} whose existence is guaranteed by Lemma \ref{lemma:1}
(here we are using that $L_n$ is self-adjoint due to the symmetry of the kernels). Next, we prove the uniqueness.
Suppose that there exist two trajectories of measures $\nu_t^n(dx)$ and $\tilde{\nu}_t^n(dx)$ such that \eqref{weak:eq2} holds. Call 
$$
\omega^n_t(dx)=\nu_t^n(dx)-\tilde{\nu}_t^n(dx).
$$ 
The evolution of $\omega^n_t(dx)$ satisfies equation \eqref{weak:eq2} with initial condition 
$$\int_{\overline\Omega}{ f}(x)\omega^n_0(dx)=0.$$
Therefore, for all $t\in [0, T]$, it holds that
\begin{align}\label{eqg}
\begin{aligned}
\int_{\overline\Omega}&f(x)\omega^n_t(dx)\\
=&\int_0^t\brc{\int_{\overline\Omega}\cor{\chi_{A_n}\pare{x}\int_{\Omega}\chi_{A_n}\pare{y} J\pare{x,y}\pare{f\pare{y}-f\pare{x}}dy+\chi_{B_n}\pare{x}\int_{\Omega}\chi_{B_n}\pare{y} G\pare{x,y}\pare{f\pare{y}-f\pare{x}}dy}\omega^n_s(dx)}ds\\[10pt]
&+\int_0^t\brc{\int_{\overline\Omega}\cor{\chi_{A_n}\pare{x}\int_\Omega \chi_{B_n}\pare{y}R\pare{x,y}\pare{f\pare{y}-f\pare{x}}dy+\chi_{B_n}\pare{x}\int_\Omega \chi_{A_n}\pare{y}R\pare{x,y}\pare{f\pare{y}-f\pare{x}}dy}\omega^n_s(dx)}ds.
\end{aligned}
\end{align}
In what follows, for all measures $\mu$ on $\overline\Omega$, we denote by $$\displaystyle{\norm{\mu}_{\text{TV}}:=\sup_{g\in C\pare{\overline\Omega}: \norm{g}_\infty\leq 1}\int_{\overline\Omega}g(x)\mu(dx)},$$ the dual norm of $\mu$ (total variation). It holds that
\begin{align}
\norm{\omega^n_t\pare{dx}}_{\text{TV}}\leq\pare{\norm{J}_\infty+2\norm{R}_\infty+\norm{G}_\infty}\pare{\abs{\Omega}+1}\int_0^t\norm{\omega^n_s\pare{dx}}_{\text{TV}}ds.
\end{align}
By Gronwell's inequality we can conclude that $\omega^n_t(dx)$ coincides with the null measure and therefore $\nu_t^n(dx)=\tilde \nu_t^n(dx)$. The uniqueness of the solution to system \eqref{weak:eq2} follows.
\end{proof}

By this uniqueness result and Lemma A.1.5.1 of \cite{KL}, we get that the process $Y_n(t)$ has a density. This is the content of the following corollary.

\begin{corollary}  \label{corol.densidad} The process $Y_n(t)$ has a density
that is characterized as the unique solution $u_n(t,x)$ to 
\begin{equation}\label{evol.fffff}
\left\{
\begin{array}{ll}
\displaystyle \frac{\partial u_n}{\partial t} (t,x) = {L}_n u_n(t,x), \qquad & \, t>0, \, x\in \Omega, \\[10pt]
u_n(0,x)=u_0(x), \qquad & x\in \overline\Omega.
\end{array}
\right.
\end{equation}
\end{corollary}

Consider now the coupled process $\pare{Y_n\pare{t}, I_n\pare{t}}\in D\pare{[0, T],\overline\Omega}\times D\pare{[0, T],\brc{1,2}}$, where
$$I_n\pare{t}=
\left\{\begin{array}{ll}
1 \qquad \mbox{if } Y_n(t) \in A_n, \\[6pt] 
2 \qquad \mbox{if } Y_n(t) \in B_n.
\end{array} \right.
$$ 
The pair $\pare{Y_n\pare{t}, I_n\pare{t}}$ is a Markov process whose generator $\mathcal{L}_n$ is defined on functions $f\in T_n$, with $$T_n:=\brc{f:\overline \Omega\times\brc{1,2}\to\mathbb R :f(\cdot, 1)\in C(A_n) \text{ and } f(\cdot, 2)\in C(B_n)},$$ as
\begin{equation}\label{gen.intro1}
\mathcal{L}_nf(x, i)=
\left\{
\begin{array}{ll} \displaystyle
\chi_{A_n}\pare{x}\int_{\Omega}\chi_{A_n}\pare{y} J\pare{x,y}\pare{f\pare{y,1}-f\pare{x,1}}dy 
+\chi_{A_n}\pare{x}\int_\Omega \chi_{B_n}\pare{y}R\pare{x,y}\pare{f\pare{y,2}-f\pare{x,1}}dy
&\quad\text{if $i=1$,} \\[10pt] 
\displaystyle \chi_{B_n}\pare{x}\int_\Omega \chi_{B_n}\pare{y}G\pare{x,y}\pare{f\pare{y, 2}-f\pare{x, 2}}dy
+\chi_{B_n}\pare{x}\int_{\Omega}\chi_{A_n}\pare{y} R\pare{x,y}\pare{f\pare{y, 1}-f\pare{x, 2}}dy   & \quad\text{if $i=2$.}
\end{array}
\right.
\end{equation}
By Lemma A.1.5.1 of \cite{KL} we know that, for every bounded function $f\in \overline\Omega\times\brc{1,2}\to \bb R$,
\begin{align}\label{f1}
M_n^f\pare{t}=f\pare{Y_n\pare{t}, I_n\pare{t}}-f\pare{Y_n\pare{0}, I_n\pare{0}}-\int_0^t\mathcal L_nf\pare{Y_n\pare{s}, I_n\pare{s}}ds
\end{align}
and
\begin{align}\label{f2}
N_n^f\pare{t}=\pare{M_n^f\pare{t}}^2-\int_0^t\pare{\mathcal L_n\pare{f\pare{Y_n\pare{s}, I_n\pare{s}}}^2-2f\pare{Y_n\pare{s}, I_n(s)}\mathcal L_nf\pare{Y_n\pare{s} I_n(s)}} ds
\end{align}
are martingales with respect to the natural filtration generated by the process.

Let $P_n\in \mathcal M_1\pare{D\pare{[0, T], \overline\Omega}\times D\pare{[0, T], \brc{1,2}}}$ be 
the law of the process $\pare{Y_n\pare{t}, I_n\pare{t}}_{t\in [0, T]}$; in 
our notation $\mathcal M_1(X)$ denotes the space of probability measures 
on a metric space $X$. The next lemma guarantees the tightness of the sequence $\pare{P_n}_ {n\in\bb N}$.
\begin{lemma}\label{lemma1}
The sequence of probability measures $\pare{P_n}_{n\in \bb N}$ is tight.
\end{lemma}
\begin{proof}
Let $P_n^1$ and $P_n^2$  be the two marginals of $P_n$. Since $D\pare{[0, T], \overline\Omega}\times D\pare{[0, T], \brc{1,2}}$ is endowed with the product topology, in order to conclude, it is enough to show that the marginals $P_n^1$ and $P_n^2$ are tight. 

We start by proving that the sequence $P_n^1$ is tight.
By Theorem 1.3 and Proposition 1.6 of Chapter 4 in \cite{KL}, it is sufficient to show that the following conditions hold:
\begin{enumerate}
\item for every $ t\in [0,T]$ and $\epsilon>0$ there exists a compact set $K(t,\epsilon )\subseteq \overline\Omega$ such that 
$$\sup_{n} P_{n}^1\Big( Y\pare{\cdot}: Y\pare{t}\not\in K\pare{t,\epsilon} \Big)\leq \epsilon,$$
\item for every $ \epsilon>0$, we have that $$\lim_{\zeta\to 0}\limsup_{n\to +\infty} \sup_{\tau\in \Lambda_T, \theta\leq\zeta}  
P_{n}^1\Big( Y\pare{\cdot}: \abs{Y\pare{\tau+\theta}-Y\pare{\tau}}>\epsilon \Big)=0,$$
where $\Lambda_T$ is the family of all stopping times bounded by $T$.
\end{enumerate}

The first condition is satisfied since $\overline{\Omega}$ is a compact space. To prove the second condition, fix $\tau\in \Lambda_T$, $\epsilon>0$ and observe that, considering the function $g\pare{x, i}=x$ in \eqref{f1}, we get that
$$
M_n^{g}\pare{t}=Y_n\pare{t}-Y_n\pare{0}-\int_0^t\mathcal L_n g\pare{Y_n\pare{s}, I_n\pare{s}}ds.
$$
Therefore,
\begin{align}\label{con1}
\abs{Y_n\pare{\tau+\theta}-Y_n\pare{\tau}}\leq \abs{\int_\tau^{\tau+\theta}\mathcal L_n  g \pare{Y_n\pare{s}, I_n\pare{s}}ds}+
\Big|M_n^{g}\pare{\tau+\theta}-M_n^{g}\pare{\tau} \Big|.
\end{align}
Since
\begin{align}
\mathcal L_ng\pare{x,i}=&\chi_{A_n}\pare{x}\chi_{1}(i)\cor{\int_{\Omega}\chi_{A_n}\pare{y} J\pare{x,y}\pare{y-x}dy +\int_\Omega \chi_{B_n}\pare{y}R\pare{x,y}\pare{y-x}dy}\\
&+ \chi_{B_n}\pare{x}\chi_2(i)\cor{\int_\Omega \chi_{B_n}\pare{y}G\pare{x,y}\pare{y-x}dy+\int_{\Omega}\chi_{A_n}\pare{y} R\pare{x,y}\pare{y-x}dy }, 
\end{align}
we have that
\begin{align}\label{c1}
\abs{\int_\tau^{\tau+\theta}\mathcal L_n g \pare{Y_n\pare{s}, I_n\pare{s}}ds}\leq C_1\theta,
\end{align}
where $C_1:=C_1\pare{\norm{J}_\infty, \norm{G}_\infty, \norm{R}_\infty, |\Omega|}$ is a constant depending on $\norm{J}_\infty, \norm{G}_\infty, \norm{R}_\infty$ and $|\Omega|$. Moreover, by \eqref{f2} we get that
$$
\bb E\pare{\pare{M_n^{g}\pare{\tau+\theta}}^2-\pare{M_n^{g}\pare{\tau}}^2}
=\bb E\pare{\int_{\tau}^{\tau+\theta} \pare{\mathcal L_n\pare{g\pare{Y_n\pare{s}, I_n\pare{s}}}^2-2g\pare{Y_n\pare{s}, I_n(s)}\mathcal L_ng\pare{Y_n\pare{s} I_n(s)}}ds}.
$$
Since
\begin{align}
&\mathcal L_n\pare{g\pare{x, i}}^2-2g\pare{x, i}\mathcal L_ng\pare{x ,i}\\
&= \chi_{A_n}\pare{x}\chi_{1}(i)\cor{\int_{\Omega}\chi_{A_n}\pare{y} J\pare{x,y}\pare{y^2-x^2}dy +\int_\Omega \chi_{B_n}\pare{y}R\pare{x,y}\pare{y^2-x^2}dy}\\
&\hspace{+25pt}+ \chi_{B_n}\pare{x}\chi_2(i)\cor{\int_\Omega \chi_{B_n}\pare{y}G\pare{x,y}\pare{y^2-x^2}dy+\int_{\Omega}\chi_{A_n}\pare{y} R\pare{x,y}\pare{y^2-x^2}dy }\\
&\hspace{+25pt}-2x\chi_{A_n}\pare{x}\chi_{1}(i)\cor{\int_{\Omega}\chi_{A_n}\pare{y} J\pare{x,y}\pare{y-x}dy +\int_\Omega \chi_{B_n}\pare{y}R\pare{x,y}\pare{y-x}dy}\\
&\hspace{+25pt} -2x \chi_{B_n}\pare{x}\chi_2(i)\cor{\int_\Omega \chi_{B_n}\pare{y}G\pare{x,y}\pare{y-x}dy+\int_{\Omega}\chi_{A_n}\pare{y} R\pare{x,y}\pare{y-x}dy }.
\end{align}
we obtain that
\begin{align}
&\bb E\pare{\pare{M_n^{g}\pare{\tau+\theta}}^2-\pare{M_n^{g}\pare{\tau}}^2}\leq C_1\theta,
\end{align}
Therefore, by Markov's inequality 
\begin{equation}\label{c2}
\bb P\pare{\abs{M_n^{g}\pare{\tau+\theta}-M_n^{g}\pare{\tau}}>\epsilon}\leq \frac{\bb E\pare{\pare{M_n^{g}\pare{\tau+\theta}}^2-\pare{M_n^{g}\pare{\tau}}^2}}{\epsilon^2} \leq\frac{C\theta}{\epsilon^2},
\end{equation}
for all $\epsilon>0$.
The bounds \eqref{con1}, \eqref{c1} and \eqref{c2} allow to conclude the second condition that guarantees the tightness of the sequence $P_n^1$.

We proceed now in a similar way to prove the tightness of the sequence $P_n^2$. As before it is enough to show that
\begin{enumerate}
\item for every $ t\in [0,T]$ and every $ \epsilon>0$ there exists a compact set $K(t,\epsilon )\subseteq\brc{1,2}$ such that 
$$\sup_{n} P_{n}^2\Big(I\pare{\cdot}: I\pare{t}\not\in K\pare{t,\epsilon}\Big)\leq \epsilon,$$
\item for every $ \epsilon>0$ it holds that $$\lim_{\zeta\to 0}\limsup_{n\to +\infty} \sup_{\tau\in L_T,\theta\leq\zeta}  
P_{n}^2\Big( I\pare{\cdot}: \abs{I\pare{\tau+\theta}-I\pare{\tau}}>\epsilon \Big) =0.$$
\end{enumerate}
The first condition is trivially satisfied taking $K(t, \epsilon)=\brc{1,2}$. Hence, we need to prove the second condition.
Considering the function $h\pare{x, i}=i$ in \eqref{f1}, we get that
$$
M_n^{h}\pare{t}=I_n\pare{t}-I_n\pare{0}-\int_0^t\mathcal L_n
h\pare{Y_n\pare{s}, I_n\pare{s}}ds.
$$
Therefore
\begin{align}\label{con0}
\abs{I_n\pare{\tau+\theta}-I_n\pare{\tau}}\leq \abs{\int_\tau^{\tau+\theta}\mathcal L_nh\pare{Y_n\pare{s}, I_n\pare{s}}ds}+
\Big| {M_n^{h}\pare{\tau+\theta}-M_n^{h}\pare{\tau}} \Big|.
\end{align}
Since
$$
\mathcal L_nh\pare{x,i}=\chi_{A_n}\pare{x} \chi_{1}\pare{i} \int_\Omega \chi_{B_n}\pare{y} R\pare{x,y}dy
-\chi_{B_n}\pare{x} \chi_{2}\pare{i} \int_\Omega \chi_{A_n}\pare{y}R\pare{x,y}dy,
$$
we have that
\begin{align}\label{d1}
\abs{\int_\tau^{\tau+\theta}\mathcal L_n h \pare{Y_n\pare{s}, I_n\pare{s}}ds}\leq C_2\theta,
\end{align}
where $C_2=C_2\pare{\norm{R}_\infty, |\Omega|}$ is a constant depending on $\norm{R}_\infty$ and $|\Omega|$. Moreover, by \eqref{f2}, we get that
\begin{align}
&\bb E\pare{\pare{M_n^{h}\pare{\tau+\theta}}^2-\pare{M_n^{h}\pare{\tau}}^2}
=\bb E\pare{\int_{\tau}^{\tau+\theta} \pare{\mathcal L_n\pare{h \pare{Y_n\pare{s}, I_n\pare{s}}}^2-2 h \pare{Y_n\pare{s}, I_n(s)}\mathcal L_n h \pare{Y_n\pare{s} I_n(s)}} ds  } \leq C_2\theta.
\end{align}
The last inequality follows from the fact that
\begin{align}
& \mathcal L_n\pare{h\pare{x, i}}^2-2 h \pare{x, i}\mathcal L_nf_2\pare{x, i} =  \pare{3-2i}\chi_{A_n}\pare{x} \chi_{1}\pare{i} \int_\Omega \chi_{B_n}\pare{y} R\pare{x,y}dy+\pare{-3+2i}\chi_{B_n}\pare{x} \chi_{2}\pare{i} \int_\Omega \chi_{A_n}\pare{y}R\pare{x,y}dy.
\end{align}
Finally, by Markov's inequality we get that
\begin{equation}\label{d2}
\bb P\pare{\abs{M_n^{h}\pare{\tau+\theta}-M_n^{h}\pare{\tau}}>\epsilon}\leq \frac{\bb E\pare{\pare{M_n^{h}\pare{\tau+\theta}}^2-\pare{M_n^{h}\pare{\tau}}^2}}{\epsilon^2} \leq\frac{C_2\theta}{\epsilon^2},
\end{equation}
for all $\epsilon>0$.
Bounds \eqref{con1}, \eqref{d1} and \eqref{d2} allow to conclude the second condition that guarantees the tightness of the sequence $P_n^2$.
\end{proof}
Lemma \ref{lemma1} guarantees that the sequence of processes $\pare{Y_n(t), I_n(t)}_{t\in [0, T]}$ converges in distribution along subsequences.
In the following theorem we prove that all subsequences converge to the same limit and we characterize the generator of the limit process.

\begin{theorem}\label{Prop:1}
The sequence $\pare{Y_n\pare{t}, I_n\pare{t}}$ converges 
$$\pare{Y_n\pare{t}, I_n\pare{t}}\xrightarrow[n\to +\infty]{D}\pare{Y\pare{t}, I\pare{t}}$$ in $D\pare{[0, T], \overline \Omega}\times D\pare{[0, T], \brc{1,2}}$. 
The limit $\pare{Y\pare{t}, I\pare{t}}$ is a Markov process 
 whose generator $\widetilde{\mathcal L}$ is defined on functions $f\in C\pare{\overline\Omega\times\brc{1,2}}$ as
\begin{align}
\widetilde{\mathcal L}f\pare{x, i}= \, &\chi_{1}\pare{i}
\brc{\int_\Omega X(y)J\pare{x, y}\pare{f\pare{y,1}-f\pare{x,1}}dy+\int_\Omega \pare{1-X(y)}R\pare{x, y}\pare{f\pare{y, 2}-f\pare{x, 1}}dy}\\[10pt]
&+\chi_{2}\pare{i}\brc{ \int_\Omega \pare{1-X\pare{y}}G\pare{x, y}\pare{f\pare{y,2}-f\pare{x, 2}}dy+\int_\Omega X(y)R\pare{x, y}\pare{f\pare{y, 1}-f\pare{x, 2}}dy}.
\end{align}
\end{theorem}

\begin{proof}
Lemma \ref{lemma1} implies that any subsequence of $P_n$ has a convergent sub-subsequence; it remains then to characterize all the limit points of the sequence $ P_n$. Let $ \tilde P$ be a limit point and $ P_{n_k}$ be a subsequence converging to $\tilde P$.
To prove the theorem it is enough to show that $\tilde P$ concentrates its mass on a process $\pare{Y\pare{\cdot}, I\pare{\cdot}}$ such that,
\begin{align}
f\pare{Y\pare{t}, I\pare{t}}-f\pare{Y\pare{0}, I\pare{0}}-\int_0^t \widetilde{\mathcal L}f\pare{Y\pare{s}, I\pare{s}}ds
\end{align}
is a martingale, for every $f\in C\pare{\overline\Omega\times\brc{1,2}, \bb R}$ and for every $t\in [0,T]$. This implies convergence of the entire sequence $P_n$ and caracterizes the limit $\tilde P$ as the law of the Markov process with generator $\widetilde{\mathcal L}$, we refer the reader to Chapter 4 in \cite{EK} for a deeper discussion of the issue.
Therefore, to conclude the proof we need to show that,
\begin{align}\label{pok0}
\bb E^{\tilde P}\cor{g\pare{\pare{Y\pare{s}, I\pare{s}}, 0\leq s\leq t_0}\pare{f\pare{Y\pare{t}, I\pare{t}}-f\pare{Y\pare{t_0}, I\pare{t_0}}-\int_{t_0}^t\widetilde{\mathcal L}f\pare{Y\pare{s}, I\pare{s}}ds}}=0,
\end{align}
for every bounded continuous function $g:\mathcal D\pare{[0, T], \overline\Omega}\times \mathcal D\pare{[0, T], \brc{1,2}}\to \bb R$,  for every $f\in C\pare{\overline\Omega\times\brc{1,2}}$ and for every $0\leq t_0<t\leq T$. 

By the tightness proved in Lemma \ref{lemma1} we know that
\begin{align}\label{pouy}
\bb E^{\tilde P}&\cor{g\pare{\pare{Y\pare{s}, I\pare{s}}, 0\leq s\leq t_0}\pare{f\pare{Y\pare{t}, I\pare{t}}-f\pare{Y\pare{t_0}, I\pare{t_0}}-\int_{t_0}^t \tilde{\mathcal L}f\pare{Y\pare{s}, I\pare{s}}ds}}\\
&=\lim_{n_k\to +\infty}\bb E^{P^{n_k}}\pare{g\pare{\pare{Y\pare{s}, I\pare{s}}, 0\leq s\leq t_0}\pare{f\pare{Y\pare{t}, I\pare{t}}-f\pare{Y\pare{t_0}, I\pare{t_0}}-\int_{t_0}^t \tilde{\mathcal L}f\pare{Y\pare{s}, I\pare{s}}ds}}.
\end{align}

By the triangular inequality, we have
\begin{align}\label{ineq10}
&\abs{\bb E^{P^{n_k}}\pare{g\pare{\pare{Y\pare{s}, I\pare{s}}, 0\leq s\leq t_0}\pare{f\pare{Y\pare{t}, I\pare{t}}-f\pare{Y\pare{t_0}, I\pare{t_0}}-\int_{t_0}^t\widetilde{\mathcal L}f\pare{Y\pare{s}, I\pare{s}}ds}} }\\[10pt]
&\hspace{+20pt}\leq\abs{\bb E^{P_{n_k}}\pare{g\pare{\pare{Y\pare{s}, I\pare{s}}, 0\leq s\leq t_0}\pare{f\pare{Y\pare{t}, I\pare{t}}-f\pare{Y\pare{t_0}, I\pare{t_0}}-\int_{t_0}^t\mathcal L_{n_k}f\pare{Y\pare{s}, I\pare{s}}ds}}}\\[10pt]
&\hspace{+35pt}+\abs{\bb E^{P_{n_k}}\pare{g\pare{\pare{Y\pare{s}, I\pare{s}}, 0\leq s\leq t_0}\int_{t_0}^t\pare{\mathcal L_{n_k}f\pare{Y\pare{s}, I\pare{s}}-\tilde{\mathcal L}f\pare{Y\pare{s}, I\pare{s}}}ds}}.
\end{align}

Let us analyze the first term in the right hand side of \eqref{ineq10}. By \eqref{f1}, we have that
\begin{align}
f\pare{Y\pare{t}, I\pare{t}}-f\pare{Y\pare{t_0}, I\pare{t_0}}-\int_{t_0}^t\mathcal L_{n_k}f\pare{Y\pare{s}, I\pare{s}}ds
\end{align}
is a martingale. Therefore, 
$$
\bb E^{P_{n_k}}\pare{g\pare{\pare{Y\pare{s}, I\pare{s}}, 0\leq s\leq t_0}\pare{f\pare{Y\pare{t}, I\pare{t}}-f\pare{Y\pare{t_0}, I\pare{t_0}}-\int_{t_0}^t\mathcal L_{n_k}f\pare{Y\pare{s}, I\pare{s}}ds}}
=0.
$$
Hence, to conclude \eqref{pok0} we just need to show that
\begin{align}
\lim_{n_k\to \infty}\abs{\bb E^{P_{n_k}}\pare{g\pare{\pare{Y\pare{s}, I\pare{s}}, 0\leq s\leq t_0}\int_{t_0}^t\pare{\mathcal L_{n_k}f\pare{Y\pare{s}, I\pare{s}}-\widetilde{\mathcal L}f\pare{Y\pare{s}, I\pare{s}}}ds}}=0.
\end{align}
Since
\begin{align}
&\abs{\bb E^{P_{n_k}}\pare{g\pare{\pare{Y\pare{s}, I\pare{s}}, 0\leq s\leq t_0}\int_{t_0}^t\pare{\mathcal L_{n_k}f\pare{Y\pare{s}, I\pare{s}}-\widetilde{\mathcal L}f\pare{Y\pare{s}, I\pare{s}}}ds}}\\[10pt]
&\hspace{+15pt}\leq\norm{g}_\infty\bb E^{P_{n_k}}\pare{\int_{t_0}^t\abs{\mathcal L_{n_k}f\pare{Y\pare{s}, I\pare{s}}-\widetilde{\mathcal L}f\pare{Y\pare{s}, I\pare{s}}}ds}
\end{align}
it is enough to prove that
\begin{align}\label{tesi30}
\lim_{n_k\to \infty}\bb E^{P_{n_k}}
\pare{\int_{t_0}^t\abs{\mathcal L_{n_k}f\pare{Y\pare{s}, I\pare{s}}-\widetilde{\mathcal L}f\pare{Y\pare{s}, I\pare{s}}}ds}=0.
\end{align}
Denoting by $\mathcal D:= \mathcal D\pare{[0, T], \overline\Omega}\times \mathcal D\pare{[0, T], \brc{1,2}}$ and using Fubini's theorem  we get
\begin{align}\label{equal1}
&\bb E^{P_{n_k}}\pare{\int_{t_0}^t\abs{\mathcal L_{n_k}f\pare{Y\pare{s}, I\pare{s}}-\widetilde{\mathcal L}f\pare{Y\pare{s}, I\pare{s}}}ds}
=\int_{t_0}^t\int_{\mathcal D}\abs{\mathcal L_{n_k}f\pare{Y\pare{s}, I\pare{s}}-\widetilde{\mathcal L}f\pare{Y\pare{s}, I\pare{s}}}dP_{n_k}\pare{Y, I}ds.
\end{align}
Observe that $I_n(s)=1+\chi_{{ B_n}}\pare{Y_n(s)}$ and $\chi_{A_n}\pare{Y_n(s)}=\chi_{1}\pare{I_n(s) }$. Then, recalling that  $u_n(x,s)$ is the probability density of the process $Y_n(s)$, we get
\begin{align}
&\int_{t_0}^t\int_{\mathcal D}\abs{\mathcal L_{n_k}f\pare{Y\pare{s}, I\pare{s}}-\widetilde{\mathcal L}f\pare{Y\pare{s}, I\pare{s}}}dP_{n_k}\pare{Y, I}ds\\
&\hspace{+15pt}=\int_{t_0}^t\int_{\Omega}\abs{\mathcal L_{n_k}f\pare{x, 1+\chi_{{B_{n_k}}}\pare{x}}-\widetilde{\mathcal L}f\pare{x, 1+\chi_{{B_{n_k}}}\pare{x}}}u_{n_k}\pare{x,s}dxds\\[10pt]
&\hspace{+15pt}\leq\int_{t_0}^t\int_\Omega\abs{\Upsilon_{n_k}^1\pare{x}+\Upsilon_{n_k}^2\pare{x}+\Upsilon_{n_k}^3\pare{x}+\Upsilon_{n_k}^4\pare{x}}u_{n_k}\pare{x, s}dxds,
\end{align}
where
\begin{align}
&\Upsilon_{n_k}^1\pare{x}=\chi_{A_{n_k}}\pare{x}\int_\Omega \chi_{A_n}\pare{y}J\pare{x,y}\pare{f\pare{y, 1}-f\pare{x, 1}}dy
- \chi_{A_{n_k}}\pare{x}\int_\Omega X(y)J\pare{x, y}\pare{f\pare{y,1}-f\pare{x,1}}dy,\\[10pt]
&\Upsilon_{n_k}^2\pare{x}=\chi_{A_{n_k}}\pare{x}\int_\Omega \chi_{B_{n_k}}\pare{y}R\pare{x,y}\pare{f\pare{y, 2}-f\pare{x, 1}}dy
- \chi_{A_{n_k}}\pare{x}\int_\Omega \pare{1-X(y)}R\pare{x, y}\pare{f\pare{y, 2}-f\pare{x, 1}}dy,\\[10pt]
&\Upsilon_{n_k}^3\pare{x}=\chi_{B_{n_k}}\pare{x}\int_\Omega \chi_{B_{n_k}}\pare{y}G\pare{x,y}\pare{f\pare{y,2}-f\pare{x,2}}dy
-\chi_{B_{n_k}}\pare{x} \int_\Omega \pare{1-X\pare{y}}G\pare{x, y}\pare{f\pare{y,2}-f\pare{x, 2}}dy,\\[10pt]
&\Upsilon_{n_k}^4\pare{x}=\chi_{B_{n_k}}\pare{x}\int_\Omega \chi_{A_{n_k}}\pare{y}R\pare{x,y}\pare{f\pare{y, 1}-f\pare{x, 2}}dy
- \chi_{B_{n_k}}\pare{x}\int_\Omega X(y)R\pare{x, y}\pare{f\pare{y, 1}-f\pare{x, 2}}dy.
\end{align}
Therefore, \eqref{tesi30} is proved once we show that
\begin{align}\label{fourp}
\lim_{n_k\to \infty}\int_{t_0}^t\int_{\Omega}\abs{\Upsilon_{n_k}^i\pare{x}}u_{n_k}( s,x)ds=0,\qquad\forall i\in \brc{1,2,3, 4}.
\end{align}
Since $u_n(s,x)\chi_{A_n}(x)=a_n(s,x)$, we get
\begin{align}\label{form1}
&\int_{t_0}^t\int_\Omega\abs{\Upsilon^1_{n_k}(x)}u_{n_k}\pare{s,x}dxds\\[10pt]
&\hspace{+0pt}=\int_{t_0}^t\int_{\Omega}\abs{\int_\Omega \chi_{A_n}\pare{y}J\pare{x,y}\pare{f\pare{y, 1}-f\pare{x, 1}}dy- \pare{x}\int_\Omega X(y)J\pare{x, y}\pare{f\pare{y,1}-f\pare{x,1}}dy}a_{n_k}\pare{ s,x}dxds\\[10pt]
&\hspace{+0pt}\leq\int_{t_0}^t\int_{\Omega}\abs{\int_\Omega \pare{\chi_{A_{n_k}}\pare{y}-X(y)}J\pare{x,y}f\pare{y,1}dy}a_{n_k}\pare{s,x}dxds
+\int_{t_0}^t\int_{\Omega}\abs{\int_\Omega \pare{\chi_{A_{n_k}}\pare{y}-X(y)}J\pare{x,y}dy}\abs{f\pare{x, 1}}a_{n_k}\pare{s,x}dxds.
\end{align}
By the continuity of $J$ and the fact that $\chi_{A_n} (x) \rightharpoonup X(x)$ (see \eqref{cond.sets}) we get
\begin{align}\label{mj1}
\sup_{x\in \Omega}\abs{\int_\Omega \Big\{\chi_{A_n}\pare{y}-X(y)\Big\}f(y, 1)J\pare{x,y}dy }\xrightarrow[n\to +\infty]{}0,
\end{align}
and
\begin{align}\label{mj2}
\sup_{x\in \Omega}\abs{\int_\Omega \Big\{\chi_{A_n}\pare{y}-X(y)\Big\}J\pare{x,y}dy}\xrightarrow[n\to +\infty]{}0.
\end{align}
Recall that $a_{n_k}(s,x)\rightharpoonup a(s,x)$ (see Theorem \ref{teo.main.intro}). By \eqref{mj1} and \eqref{mj2}, we obtain that
the right hand side of \eqref{form1} converges to $0$ as $n\to +\infty$.
Arguing as before we can conclude that
\begin{align}\label{form2}
\int_{t_0}^t\int_\Omega\abs{\Upsilon^i_{n_k}(x)}u_{n_k}\pare{s,x}dxds\xrightarrow[n_k\to\infty]{}0,\qquad\forall i\in \brc{2,3,4}.
\end{align}
This concludes the proof of \eqref{fourp}.
\end{proof}

We  can prove now the last statement of Theorem \ref{teo.main.intro}, i.e., that the distribution of the limit process $\pare{Y(t), I(t)}$ is characterized by the densities $a(t,x)$ and $b(t,x)$.

First of all, observe that, for every measurable $E\subseteq \overline\Omega$,
\begin{align}
\bb P\pare{\pare{Y_n(0), I_n(0)}\in E\times\brc{1}}&=\bb P\pare{Y_n(0)\in E\cap A_n}\\
&=\int_{E\cap A_n}u_0(x)dx\\
&=\int_E u_0(x)\chi_{A_n}(x)dx\xrightarrow[n\to\infty]{}\int_E u_0(x)X\pare{x}dx.
\end{align}
Therefore, by the tightness result proved in Lemma \ref{lemma1}, we can write
\begin{align}\label{ic1}
\bb P\pare{\pare{Y(0), I(0)}\in E\times\brc{1}}=\int_E u_0(x)X\pare{x}dx.
\end{align}
Analogously, we get that
\begin{align}\label{ic2}
\bb P\pare{\pare{Y(0), I(0)}\in E\times\brc{2}}=\int_E u_0(x)\pare{1-X\pare{x}}dx.
\end{align}

Let $\mu(dx, i)=\pare{\mu_t(dx, i)}_{t\in [0, T]}$ be the law of the limit process $\pare{Y(t), I(t)}_{t\in [0, T]}$. We can decompose
\begin{align}\label{decmea}
\mu_t(dx, i)=\bb \chi_{1}(i)\mu_t(dx, 1)+\bb \chi_{2}(i)\mu_t(dx, 2)
\end{align}
where, by \eqref{ic1} and \eqref{ic2}, $\pare{\mu_t(dx, 1)}_{t\in [0, T]}$ and $\pare{\mu_t(dx, 2)}_{t\in [0, T]}$ are such that
\begin{align}\label{ic}
\mu_0(dx, 1)=u_0(x)X\pare{x}dx \quad \text{ and }\quad \mu_0(dx, 2)=u_0(x)\pare{1-X(x)}dx.
\end{align}
Since $\widetilde{\mathcal L}$ is the generator of the process $\pare{Y(t), I(t)}$ (see Theorem \ref{Prop:1}), by Lemma A.5.1 of \cite{KL} we can conclude that 
\begin{align}\label{endp}
&\displaystyle \frac{\partial}{\partial t}\sum_{i=1}^2\int_{\overline\Omega}f(x, i)\mu_t(dx, i)=\sum_{i=1}^2\int_{\overline\Omega}\widetilde{\mathcal L}f(x, i)\mu_t(dx, i),
\end{align}
for all bounded $f:\overline\Omega\times \brc{1,2}\to \bb R$.
Therefore, fixing $g\in C\pare{\overline\Omega}$ and choosing $f(x, i)=g(x)\chi_{1}(i)$, we get
\begin{align}\label{eqm1}
\frac{\partial }{\partial t}\int_{\overline\Omega}g(x)\mu_t(dx, 1)&=\int_{\overline\Omega}\int_{\overline\Omega}X(y)J\pare{x, y}\pare{g\pare{y}-g\pare{x}}dy\;\mu_t(dx,1)\\[10pt]
&\qquad -\int_{\overline\Omega}\int_{\overline\Omega}\pare{1-X\pare{y}}R\pare{x, y}g\pare{x}dy\;\mu_t\pare{dx, 1}\\[10pt]
&\qquad +\int_{\overline\Omega}\int_{\overline\Omega}X(y) R\pare{x, y}g\pare{y}dy\;\mu_t(dx, 2).
\end{align}
Choosing  $f(x, i)=g(x)\chi_{2}(i)$ we get
\begin{align}\label{eqm2}
\frac{\partial }{\partial t} \int_{\overline\Omega}g(x)\mu_t(dx, 2)&=\int_{\overline\Omega}\int_{\overline\Omega}\pare{1-X\pare{y}}R\pare{x, y}g\pare{y}dy\;\mu_t\pare{dx, 1}\\
&\qquad +\int_{\overline\Omega}\int_{\overline\Omega}\pare{1-X\pare{y}}G\pare{x, y}\pare{g(y)-g\pare{x}}dy\;\mu_t(dx, 2)\\
&\qquad-\int_{\overline\Omega}\int_{\overline\Omega}X\pare{y}R\pare{x, y}g\pare{x}dy\;\mu_t(dx, 2).
\end{align}
We analyse now the right hand side of \eqref{eqm1}. Since $J$ is symmetric, by a change of variables, we can write
\begin{align}\label{adj1}
\int_{\overline\Omega}\int_{\overline\Omega}X(y)J\pare{x, y}\pare{g\pare{y}-g\pare{x}}dy\;\mu_t\pare{dx, 1} &=\int_{\overline\Omega}\int_{\overline\Omega}X(x)J(x,y)g(x)dx\;\mu_t\pare{dy, 1}\\
 &\hspace{+15pt}-\int_{\overline\Omega}\int_{\overline\Omega}X(y)J(x,y)g(x)dy\;\mu_t\pare{dx, 1}.
\end{align}
Moreover, it holds that
\begin{align}\label{adj3}
&\int_{\overline\Omega}\int_{\overline\Omega}X\pare{y} R\pare{x, y}g\pare{y}dy\;\mu_t(dx, 2)=\int_{\overline\Omega}\int_{\overline\Omega}X(x) R\pare{x, y}g\pare{x}dx\;\mu_t(dy, 2).
\end{align}
Replacing \eqref{adj1} and \eqref{adj3} in the right hand side of \eqref{eqm1} we obtain
\begin{align}\label{wqer}
\frac{\partial }{\partial t} \int_{\overline\Omega}g(x)\mu_t(dx, 1)&=\int_{\overline\Omega}\int_{\overline\Omega}g\pare{x}J(x,y)X(x)dx\;\mu_t\pare{dy, 1}
-\int_{\overline\Omega}\int_{\overline\Omega}g(x)J(x,y)X(y)dy\;\mu_t\pare{dx, 1}\\[10pt]
 &\hspace{+15pt}-\int_{\overline\Omega}\int_{\overline\Omega}g(x)R(x,y)\pare{1-X(y)}dy\;\mu_t\pare{dx, 1}
 +\int_{\overline\Omega}\int_{\overline\Omega}g\pare{x}R\pare{x, y}X(x)dx\;\mu_t(dy, 2).
\end{align}
As before, via a change of variable in the right hand side of \eqref{eqm2}, we can write
\begin{align}\label{adj2}
\frac{\partial }{\partial t} \int_{\overline\Omega}g(x)\mu_t(dx, 2)&=\int_{\overline\Omega}\int_{\overline\Omega}g(x)\pare{1-X\pare{x}}R\pare{x, y}dx\;\mu_t\pare{dy, 1}+\int_{\overline\Omega}\int_{\overline\Omega}g\pare{x}\pare{1-X\pare{x}} G\pare{x, y}dx\;\mu_t(dy, 2)\\
&\hspace{+15pt}-\int_{\overline\Omega}\int_{\overline\Omega}g\pare{x}\pare{1-X\pare{y}} G\pare{x, y}dy\;\mu_t(dx, 2)
-\int_{\overline\Omega}\int_{\overline\Omega}g\pare{x}X(y) R\pare{x, y}dy\;\mu_t(dx, 2).
\end{align}
Moreover, by \eqref{ic} we get that
\begin{align}\label{ic3}
\int_{\overline\Omega}g(x)\mu_0(dx, 1)=\int_{\overline\Omega}g(x)u_0(x)X(x)dx\quad \text{and} \quad \int_{\overline\Omega}g(x)\mu_0(dx, 2)=\int_{\overline\Omega}g(x)u_0(x)\pare{1-X\pare{x}}dx.
\end{align}
By Lemma \ref{lemma:weaksys} below we know that there exists a unique pair of trajectories of measures $\pare{\mu_t\pare{dx, 1}, \mu_t\pare{dx, 2}}_{t\in [0,T]}$ which, for every $g\in C\pare{\overline\Omega}$, satisfies \eqref{wqer}, \eqref{adj2} and \eqref{ic2}.
Such pair is given by $$\mu_t\pare{dx, 1}=a(t,x)dx \qquad \mbox{ and } \qquad 
\mu_t\pare{dx, 2}=b( t,x)dx,$$ where the couple $\pare{a(t,x), b(t,x)}$ is the unique solution to system \eqref{sys1.intro}. This concludes the proof of Theorem \ref{teo.main.intro}.

\begin{lemma}\label{lemma:weaksys}
There exists a unique pair $\pare{\mu_t\pare{dx, 1}, \mu_t\pare{dx, 2}}$ such that, for every $g\in C\pare{\overline\Omega}$, \eqref{wqer}, \eqref{adj2} and \eqref{ic3} are satisfied.
Such solution is given by $$\pare{\mu_t\pare{dx, 1}, \mu_t\pare{dx, 2}}=\pare{a( t,x)dx, b( t,x)dx},$$ where the pair $\pare{a(t,x), b(t,x)}$ is the unique solution to system \eqref{sys1.intro}.
\end{lemma}

\begin{proof}
The fact that the pair $\pare{a(t,x)dx, b(t,x)dx}$ is a solution to the system \eqref{wqer}--\eqref{adj2}--\eqref{ic3} is a consequence of 
the fact that, by Theorem \ref{teo.main.intro}, $\pare{a(t,x), b(t,x)}$ is a solution to the system \eqref{sys1.intro}. 

We prove now the uniqueness. Suppose that there exist two pairs $\pare{\nu_t\pare{dx, 1}, \nu_t\pare{dx, 2}}$ and $\pare{\tilde\nu_t\pare{dx, 1}, \tilde\nu_t\pare{dx, 2}}$ for which system \eqref{wqer} --\eqref{adj2}--\eqref{ic3} is satisfied. Let 
$$
\omega_t(dx, 1):=\nu_t\pare{dx, 1}-\tilde\nu_t\pare{dx, 1}\qquad \mbox{ and }\qquad \omega_t(dx, 2):=\nu_t\pare{dx, 2}-\tilde\nu_t\pare{dx, 2}.
$$
Therefore, we know that, for all $g\in C(\overline\Omega)$, 
\begin{align}\label{sys13}
\displaystyle \int_{\overline\Omega}g(x)\omega_t(dx, 1)&=\int_0^t\int_{\overline\Omega}\int_{\overline\Omega}g\pare{x}J(x,y)X(x)dx\;\omega_s\pare{dy, 1}ds\
-\int_0^t\int_{\overline\Omega}\int_{\overline\Omega}g(x)J(x,y)X(y)dy\;\omega_s\pare{dx, 1}ds\\
 &\hspace{+15pt}-\int_0^t\int_{\overline\Omega}\int_{\overline\Omega}g(x)R(x,y)\pare{1-X(y)}dy\;\omega_s\pare{dx, 1}ds
 +\int_0^t\int_{\overline\Omega}\int_{\overline\Omega}g\pare{x}R\pare{x, y}X(x)dx\;\omega_s(dy, 2)ds
\end{align}
and
\begin{equation}\label{sys14}
\begin{array}{l}
\displaystyle 
\int_{\overline\Omega}g(x)\omega_t(dx, 2)=\int_0^t\int_{\overline\Omega}\int_{\overline\Omega}g(x)\pare{1-X\pare{x}}R\pare{x, y}dx\omega_s\pare{dy, 1}ds+\int_0^t\int_{\overline\Omega}\int_{\overline\Omega}g\pare{x}\pare{1-X\pare{x}} G\pare{x, y}dx\omega_s(dy, 2)ds\\[10pt]
\displaystyle \hspace{+80pt}-\int_0^t\int_{\overline\Omega}\int_{\overline\Omega}g\pare{x}\pare{1-X\pare{y}} G\pare{x, y}dy\;\omega_s(dx, 2)ds
-\int_0^t\int_{\overline\Omega}\int_{\overline\Omega}g\pare{x}X(y) R\pare{x, y}dy\;\omega_s(dx, 2)ds,
\end{array}
\end{equation}
with initial conditions
\begin{align}\label{sys15}
\int_{\overline\Omega}g(x)\omega_0(dx, 1)=\int_{\overline\Omega}g(x)\omega_0(dx, 2)=0.
\end{align}

Recalling that we denote by $\norm{\cdot}_{\text{TV}}$ the dual norm (total variation), from \eqref{sys13} and \eqref{sys15} we get
\begin{align}
& \displaystyle \norm{\omega_t(dx, 1)}_{\text{TV}}  =\sup_{g\in C\pare{\overline\Omega}: \norm{g}_\infty\leq 1}\bigg\{\int_0^t\int_{\overline\Omega}\int_{\overline\Omega}g\pare{x}J(x,y)X(x)dx\;\omega_s\pare{dy, 1}ds
-\int_0^t\int_{\overline\Omega}\int_{\overline\Omega}g(x)J(x,y)X(y)dy\;\omega_s\pare{dx, 1}ds\\
 &\hspace{+100pt}-\int_0^t\int_{\overline\Omega}\int_{\overline\Omega}g(x)R(x,y)\pare{1-X(y)}dy\;\omega_s\pare{dx, 1}ds
 +\int_0^t\int_{\overline\Omega}\int_{\overline\Omega}g\pare{x}R\pare{x, y}X(x)dx\;\omega_s(dy, 2)ds\bigg\}\\
&\hspace{+60pt}\leq C_1\int_0^t\pare{\norm{\omega_s(dx, 1)}_{\text{TV}}+\norm{\omega_s(dx, 2)}_{\text{TV}}}ds ,
\end{align}
where $C_1=C_1\pare{\norm{J}_\infty,\norm{R}_\infty, \abs{\Omega}}$ is a constant depending on $\norm{J}_\infty,\norm{R}_\infty$ and $\abs{\Omega}$. Analogously, by \eqref{sys14} and \eqref{sys15}, we obtain
\begin{align}\label{gron2}
\displaystyle \norm{\omega_t(dx, 2)}_{\text{TV}}\leq C_2\int_0^t\pare{\norm{\omega_s(dx, 1)}_{\text{TV}}+\norm{\omega_s(dx, 2)}_{\text{TV}}}ds,
\end{align}
where $C_2=C_2\pare{\norm{G}_\infty,\norm{R}_\infty, \abs{\Omega}}$ is a constant depending on $\norm{G}_\infty,\norm{R}_\infty$ and $\abs{\Omega}$.
Therefore by Gronwall's inequality we conclude that
\begin{align}
\displaystyle \norm{\omega_t(dx, 1)}_{\text{TV}}+\norm{\omega_t(dx, 2)}_{\text{TV}}=0.
\end{align}
Therefore, $\omega_t(dx, 1)$ and $\omega_t(dx, 2)$ coincide with the null measure on $\overline \Omega$ and consequently $\nu_t(dx, i)=\tilde \nu_t(dx, i)$, for $i\in \brc{1,2}$. This concludes the proof.
\end{proof}

\begin{remark} {\rm It holds that
$$\pare{Y_n(t)}_{t\in [0, T]}\xrightarrow[n\to \infty]{D}\pare{Y(t)}_{t\in [0, T]},$$ where $Y(t)$ has probability density $$u(t,x)=a(t,x)+b(t,x).$$ 
Indeed, the convergence in distribution to the process $\pare{Y(t)}_{t\in [0, T]}$ is a consequence of \eqref{convd}. 
Moreover, since $Y(t)$ is the marginal in the first variable of $\pare{Y(t), I(t)}$, we can write
\begin{align}
\displaystyle \bb P\pare{Y(t)\in E}&=\sum_{i=1}^2\bb P\pare{Y(t)\in E, I(t)=i}=\sum_{i=1}^2\int_E\mu_t\pare{dx, i}=\int_E\pare{a(t,x)+b(t,x)}dx
=\int_E u(t,x) dx,
\end{align}
for every measurable set $E\subseteq \overline\Omega$.}
\end{remark}

\subsection{Asymptotic behavior of $u_n(t,x)$.}

This subsection contains the proof of the fact that $u_n(t,x)$ converges exponentially fast to $\frac{1}{|\Omega|}$. 

\begin{proof}[Proof of Theorem \ref{teo.comport.intro0}] Let $u_n$ be the solution to \eqref{evol.intro}
and define $$w_n (t,x):= u_n(t,x) - \frac{1}{|\Omega|}.$$ Observe that $w_n$ is a solution to
\begin{equation}\label{evol.intro3}
\left\{
\begin{array}{ll}
\displaystyle \frac{\partial w_n}{\partial t} (t,x) = {L}_n w_n(t,x), \qquad & \, t>0, \, x\in \overline\Omega, \\[6pt]
\displaystyle w_n(0,x)=u_0(x)-\frac{1}{\abs{\Omega}},\qquad & x\in \overline\Omega.
\end{array}
\right.
\end{equation}
Therefore we have that $$\int_\Omega w_n(0,x) \, dx =0.$$
Since our equation preserves the total mass, we have
$$
\int_\Omega w_n(t,x) \, dx =0, \qquad \forall t \geq 0.
$$

Now, multiply by $w_n$ both sides in \eqref{evol.intro3} and integrate in $\Omega$ to obtain
\begin{align}\label{parteq}
\frac{\partial }{\partial t}\frac12 \int_\Omega |w_n(t,x)|^2 \, dx &=
\int_\Omega L_n w_n (t,x) w_n(t,x) \, dx\\
&=- 2 E_n(w_n),
\end{align}
with
\begin{align}\label{ener}
\begin{array}{rl}
	\displaystyle E_n(w) = & \displaystyle \frac14 \int_{A_n}  \int_{A_n} J (x,y) (w (y) - w (x))^2 dy \, dx  
	+ \frac14 \int_{B_n}  \int_{B_n} G (x,y) (w (y) - w (x))^2 dy \, dx 
	\\[10pt]
	& \displaystyle 
	+ \frac12 \int_{A_n}  \int_{B_n} R (x,y) (w (y) - w (x))^2 dy \, dx .
	\end{array}
\end{align}
By Lemma 4.1 in \cite{nosotros} we know that there exists a positive constant $c$ (independent of $n$) such that
$$
E_n(w_n)\geq c \int_\Omega w_n^2 (t,x)\, dx .
$$
Therefore by \eqref{parteq} we obtain
	$$
\frac{\partial }{\partial t} \int_\Omega |w_n(t,x)|^2 \, dx\leq - 2 c \int_\Omega w_n^2 (t,x)\, dx.
$$
Then, by Gronwall's inequality, we conclude that there is a constant $c$ independent of $n$ such that
	\begin{align}\label{exp.ine}
	\norm{ w_n^2 (t,\cdot)}_{L^2(\Omega)}^2 \leq \|w_0\|_{L^2(\Omega)}^2 e^{- 2 ct}
	\end{align}
and this concludes the proof of Theorem \ref{teo.comport.intro0}. \end{proof}

Using Theorem \ref{teo.comport.intro0} we can prove the following proposition.
\begin{proposition}
For every $f\in C(\Omega)$ it holds that
\begin{align}\label{limk}
\mathbb{E} \Big( \int_0^{\infty} f(Y_n\pare{t}) \, dt \Big)\xrightarrow[n\to\infty]{}\mathbb{E} \Big( \int_0^{\infty} f(Y\pare{t})\, dt\Big).
\end{align}
\end{proposition}

\begin{proof}
Since $u_n(t,x)$ and $u(t,x)$ are the probability densities of $Y_n(t)$ and $Y(t)$ respectively it is enough to show that
\begin{align}\label{goal0}
\abs{\int_0^\infty\int_{\Omega}f(y)u_n( ,y)dydt-\int_0^\infty\int_{\Omega}f(y)u(t,y)dydt}\xrightarrow[n\to\infty]{}0.
\end{align}
For every fixed $T>0$ it holds that
\begin{align}
\begin{aligned}
&\abs{\int_0^\infty\int_{\Omega}f(y)u_n(t,y)dydt-\int_0^\infty\int_{\Omega}f(y)u(t,y)dydt}\\
&\leq
\abs{\int_0^T\int_{\Omega}f(y)u_n(t,y)dydt-\int_0^T\int_{\Omega}f(y)u(t,y)dydt}+\abs{\int_T^\infty\int_{\Omega}f(y)u_n(t,y)
dydt-\int_T^\infty\int_{\Omega}f(y)u(t,y)dydt}.
\end{aligned}
\end{align}
By Theorem \ref{teo.main.intro} we know that $u_n(t,x)\rightharpoonup u(t,x)$ in $L^2 ( (0,T)\times \Omega)$ and therefore
\begin{align}\label{dom0}
\abs{\int_0^T\int_{\Omega}f(y)u_n(t,y)dydt-\int_0^T\int_{\Omega}f(y)u(t,y)dydt}\xrightarrow[n\to\infty]{}0.
\end{align}
Thus to conclude \eqref{goal0} it is enough to show that
\begin{align}\label{goal1}
\lim_{T\to\infty}\lim_{n\to\infty}\abs{\int_T^\infty\int_{\Omega}f(y)u_n(t,y)dydt-\int_T^\infty\int_{\Omega}f(y)u(t,y)dydt}=0.
\end{align}

To this end, we observe that
\begin{align}\label{goal2}
\begin{aligned}
&\abs{\int_T^\infty\int_{\Omega}f(y)u_n(t,y)dydt-\int_T^\infty\int_{\Omega}f(y)u(t,y)dydt}\\
&\leq \int_T^\infty\abs{\int_{\Omega}f(y)\pare{u_n(t,y)dy-u(y, t)}dy}dt\\
&\leq \int_T^\infty\abs{\int_\Omega f(y)\pare{u_n(t,y)-\frac{1}{\abs{\Omega}}}dy}
+\int_t^\infty\abs{\int_\Omega f(y)\pare{\frac{1}{\abs{\Omega}}-u(t,y)}dy}\\
&\leq \norm{f}_{L^2(\Omega)}^2\pare{\int_T^\infty \norm{u_n(t,\cdot)-\frac{1}{\abs{\Omega}}}^2_{L^2(\Omega)}dt
+  \int_T^\infty \norm{u( t,\cdot)-\frac{1}{\abs{\Omega}}}^2_{L^2(\Omega)} dt}.
\end{aligned}
\end{align}
Let $C$ and $A$ be the constants that appear in Theorem \ref{teo.comport.intro0}, we have that 
\begin{align}\label{inj}
\norm{u_n(t,\cdot)-\frac{1}{\abs{\Omega}}}^2_{L^2(\Omega)}\leq Ce^{-At},
\end{align}
and therefore 
\begin{align}\label{ser2}
\int_T^\infty \norm{u_n(t,\cdot)-\frac{1}{\abs{\Omega}}}^2_{L^2(\Omega)}dt\leq \int_T^\infty Ce^{-At}dt\xrightarrow[T\to\infty]{}0.
\end{align}

Now, using that 
$$
u_n(t,\cdot)-\frac{1}{\abs{\Omega}} \rightharpoonup u (t,\cdot)-\frac{1}{\abs{\Omega}} \qquad \mbox{ weakly in }L^2(\Omega)
$$
as $n\to \infty$,
taking the limit as $n\to\infty$ in \eqref{inj} we get
\begin{align}\label{ser3}
\norm{u(t,\cdot)-\frac{1}{\abs{\Omega}}}_{L^2(\Omega)}^2\leq \liminf_{n\to \infty} 
\norm{u_n (t,\cdot)-\frac{1}{\abs{\Omega}}}_{L^2(\Omega)}^2 \leq Ce^{-At}
\end{align}
and therefore
\begin{align}
\int_T^\infty\norm{u(t,\cdot)-\frac{1}{\abs{\Omega}}}_{L^2(\Omega)}^2dt\xrightarrow[T\to\infty]{}0.
\end{align}

From \eqref{goal2}, \eqref{ser2} and \eqref{ser3} we conclude \eqref{goal1}.
\end{proof}

\section{Initial conditions $u_0=\delta_{\bar{x}}$.} \label{sect-deltas}
In this section, we analyze now the case in which $Y_n(0)=\bar x\in\Omega$.

Let $P_n$ be the law of the process $\pare{Y_n(t), I_n(t)}_{t\geq 0}$ and
call $\pare{\nu_n(t)}_{t}$ the law of $\pare{Y_n(t)}_t$ that is the first marginal of $P_n$.

By Dynkin's formula we know that, for every $g\in C(A_n)\cap C(B_n)$, it holds that
\begin{align}
\begin{cases}
\displaystyle\frac{d}{dt}\int_\Omega g(x)\nu_t^n(dx)=\int_\Omega L_n g(x)\nu_t^n(dx)\\[6pt]
\displaystyle\int_\Omega  g (x)\nu_0^n(dx)= g (\bar x),
\end{cases}
\end{align}
where $L_n$ is the generator defined in \eqref{gen.intro1}. Since the evolution problem does not have
a regularizing effect we expect that the initial measure $\delta_{\bar x}$ remains as time evolves, hence 
we write 
\begin{align}\label{decnu}
\nu_t^n(dx):=z_n(t, x)dx+\sigma_n(t)\delta_{\bar x}(dx).
\end{align}
By the expression of $L_n$ we obtain
\begin{align}
&\frac{d}{dt}\int_\Omega g(x)z_n(t, x)dx +\frac{d}{dt}\sigma_n(t)g(\bar x)\\
&=\int_\Omega\chi_{A_n}(x)\int_\Omega\chi_{A_n}(y) J(x, y)\pare{g(y)-g(x)}z_n(t, x)dx+\sigma_n(t)\chi_{A_n}(\bar x)\int_\Omega\chi_{A_n}(y) J(\bar x, y)\pare{g(y)-g(\bar x)}dy\\
&\qquad +\int_\Omega\chi_{B_n}(x)\int_\Omega\chi_{B_n}(y) G(x, y)\pare{g(y)-g(x)}z_n(t, x)dx+\sigma_n(t)\chi_{B_n}(\bar x)\int_\Omega \chi_{B_n}(y)G(\bar x, y)\pare{g(y)-g(\bar x)}dy\\
&\qquad +\int_\Omega\chi_{A_n}(x)\int_\Omega\chi_{B_n}(y) R(x, y)\pare{g(y)-g(x)}z_n(t, x)dx+\sigma_n(t)\chi_{A_n}(\bar x)\int_\Omega \chi_{B_n}(y)R(\bar x, y)\pare{g(y)-g(\bar x)}dy\\
&\qquad  +\int_\Omega\chi_{B_n}(x)\int_\Omega\chi_{A_n}(y) R(x, y)\pare{g(y)-g(x)}z_n(t, x)dx+\sigma_n(t)\chi_{B_n}(\bar x)\int_\Omega \chi_{A_n}(y)R(\bar x, y)\pare{g(y)-g(\bar x)}dy.
\end{align}
Therefore, we get that
\begin{align}
\frac{d}{dt}\sigma_n(t)=&-\sigma_n(t)\Bigg(\chi_{A_n}(\bar x)\int_\Omega\chi_{A_n}(y) J(\bar x, y)dy
+\chi_{B_n}(\bar x)\int_\Omega\chi_{B_n}(y) G(\bar x, y)dy\\
&\hspace{+55pt}+\chi_{A_n}(\bar x)\int_\Omega\chi_{B_n}(y) R(\bar x, y)dy
+\chi_{B_n}(\bar x)\int_\Omega\chi_{A_n}(y) R(\bar x, y)dy
\Bigg),
\end{align}
with initial datum $\sigma_n(0)=1$. Now, recall that we assumed that
\begin{align}
&\chi_{A_n}(\bar x)\int_\Omega\chi_{A_n}(y) J(\bar x, y)dy+\chi_{B_n}(\bar x)\int_\Omega\chi_{B_n}(y) G(\bar x, y)dy
+\chi_{A_n}(\bar x)\int_\Omega\chi_{B_n}(y) R(\bar x, y)dy
+\chi_{B_n}(\bar x)\int_\Omega\chi_{A_n}(y) R(\bar x, y)dy=1,
\end{align}  
for every $\bar x\in\Omega$.

This condition has a clear probabilistic interpretation. 
It says that the particle has to jump with full probability (that is, the probability of staying at the same location when
the exponential clock rings is zero). In fact, assume, for example, that $\bar x\in A_n$, then 
$$\int_\Omega\chi_{A_n}(y) J(\bar x, y)dy$$
is the probability to jump to a new position in $A_n$ and 
$$
\int_\Omega\chi_{B_n}(y) R(\bar x, y)dy
$$
gives the probability to jump to $B_n$. Then,
$$
\int_\Omega\chi_{A_n}(y) J(\bar x, y)dy+\int_\Omega\chi_{B_n}(y) R(\bar x, y)dy 
$$
is the probability to jump to a new position in $\Omega$ (and we have that it is equal to 1 since the particle is obliged to jump). 
A similar analysis can be done when $\bar x\in B_n$.

Therefore, we conclude that
\begin{align}\label{due}
\sigma_n(t)= e^{-t}, \quad \forall n\in\mathbb N.
\end{align}

On the other hand, we get
\begin{align}
\hspace{-20pt}\frac{d}{dt}\int_\Omega g(x)z_n(t, x)dx=&\int_\Omega\chi_{A_n}(x)\int_\Omega\chi_{A_n}(y) J(x, y)\pare{g(y)-g(x)}z_n(t, x)dx+\sigma_n(t)\chi_{A_n}(\bar x)\int_\Omega\chi_{A_n}(y) J(\bar x, y)g(y)dy\\
&\hspace{-15pt}+\int_\Omega\chi_{B_n}(x)\int_\Omega\chi_{B_n}(y) G(x, y)\pare{g(y)-g(x)}z_n(t, x)dx+\sigma_n(t)\chi_{B_n}(\bar x)\int_\Omega\chi_{B_n}(y) G(\bar x, y)g(y)dy\\
&\hspace{-15pt}+\int_\Omega\chi_{A_n}(x)\int_\Omega\chi_{B_n}(y) R(x, y)\pare{g(y)-g(x)}z_n(t, x)dx+\sigma_n(t)\chi_{A_n}(\bar x)\int_\Omega\chi_{B_n}(y) R(\bar x, y)g(y)dy\\
&\hspace{-15pt}+\int_\Omega\chi_{B_n}(x)\int_\Omega\chi_{A_n}(y) R(x, y)\pare{g(y)-g(x)}z_n(t, x)dx+\sigma_n(t)\chi_{B_n}(\bar x)\int_\Omega\chi_{A_n}(y) R(\bar x, y)g(y)dy,
\end{align}
which implies the following equation for $z_n$, 
\begin{align}\label{dom1}
\frac{d}{dt}z_n(t, x)=&\chi_{A_n}(x)\int_\Omega\chi_{A_n}(y) J(x, y)\pare{z_n(t, y)-z_n(t, x)}dy+\sigma_n(t)\chi_{A_n}(\bar x) \chi_{A_n}(x)J(\bar x, x)\\
&+\chi_{B_n}(x)\int_\Omega\chi_{B_n}(y) G(x, y)\pare{z_n(t, y)-z_n(t, x)}dy+\sigma_n(t)\chi_{B_n}(\bar x)\chi_{B_n}(x)G(\bar x, x)\\
&+\chi_{A_n}(x)\int_\Omega\chi_{B_n}(y) R(x, y)\pare{z_n(t, y)-z_n(t, x)}dy+\sigma_n(t)\chi_{A_n}(\bar x) \chi_{B_n}(x)R(\bar x, x)\\
&+\chi_{B_n}(x)\int_\Omega\chi_{A_n}(y) R(x, y)\pare{z_n(t, y)-z_n(t, x)}dy+\sigma_n(t)\chi_{B_n}(\bar x) \chi_{A_n}(x)R(\bar x, x),
\end{align}
with initial condition $z_n(0,x)=0$.

\subsection{Convergence along subsequences.}
We devote this subsection to the proof of Theorem \ref{teo.main.intro2}.

\begin{proof}[Proof of Theorem \ref{teo.main.intro2}]
The sequence $z_n$ converges weakly in $L^2\pare{ [0, T] \times \Omega}$ along subsequences as it is bounded the $L^2$-norm. The same holds for $\chi_{A_{n}}(x)z_{n}(t,x)$ and $\chi_{B_{n}}(x)z_{n}(t,x)$. This fact can be easily obtained working as in Lemma \ref{lema-3-2-}. 

Take $\chi_{A_{\nk}}(x)z_{\nk}(t,x)$ and $\chi_{B_{\nk}}(x)z_{\nk}(t,x)$ two convergent subsequences. 
We have to distinguish between two cases:

{\it Case 1.}
There exists a sub-subsequence $z_{n_{k_j}}$ such that 
\begin{align}
\chi_{A_{n_{k_j}}}\pare{\bar x}=1,\qquad\forall n_{k_j}.
\end{align}
We call $a_k(t,x)$ and $b_k(t,x)$ the weak limits of $\chi_{A_{\nkj}}(x)z_{\nkj}(t,x)$ and $\chi_{B_{\nkj}}(x)z_{\nkj}(t,x)$ respectively. 
Observe that $a_k(t,x)$ and $b_k(t,x)$ coincide with the weak limits of $\chi_{A_{\nk}}(x)z_{\nk}(t,x)$ and $\chi_{B_{\nk}}(x)z_{\nk}(t,x)$, respectively, as we know that the two sequences converge along subsequences.

Take now a smooth function $\phi$ such that $\phi(T,\cdot)\equiv 0$ and consider equation \eqref{dom1}. Multiply both sides by $\chi_{B_n}\pare{x}\phi\pare{t,x}$ and then integrate respect to the variables $x$
and $t$. Since by construction $\phi(T,\cdot)\equiv 0$, integrating by parts we obtain
\begin{align}
-\int_0^T\int_\Omega &\frac{\partial \phi}{\partial t}(t,x)b_{\nkj}\pare{t,x}\, dxdt\\
&\hspace{-25pt}=\int_0^T\int_\Omega\int_{\Omega}\chi_{B_{\nkj}}\pare{x}\chi_{B_{\nkj}}\pare{y} G\pare{x,y}\pare{z_{\nkj}\pare{t,y}-z_{\nkj}\pare{t,x}}\phi\pare{t,x}dydxdt\\
&+\int_0^T\int_\Omega\int_{\Omega}\chi_{B_{\nkj}}\pare{x}\chi_{A_{\nkj}}\pare{y}R\pare{x,y}\pare{z_{\nkj}\pare{t,y}-z_{\nkj}\pare{t,x}}
\phi\pare{t,x}dydxdt\\
&+\int_0^T\int_\Omega e^{-t}R(\bar x, x)\chi_{B_{\nkj}}(x)\phi\pare{ t,x}dx dt.
\end{align}
We can analyze the terms in the left hand side and the first two terms in the right hand side of the previous equality exactly as we did in the proof of Theorem \ref{teo.main.intro} (see \eqref{formula1}, \eqref{formula3}, \eqref{formula4} and \eqref{formula5}). Moreover, we have that
\begin{align}
\int_0^T\int_\Omega e^{-t}G(\bar x, x)\chi_{B_{\nkj}}(x)\phi\pare{t,x}dx dt\xrightarrow[n\to\infty]{}\int_0^T\int_\Omega e^{-t}G(\bar x, x)\pare{1-X(x)}\phi\pare{ t,x}dx dt.
\end{align}

Therefore, we get
\begin{align}\label{eqp1}
\displaystyle\frac{\partial b_j}{\partial t}\pare{t,x}=&\int_\Omega G\pare{x, y}\pare{\pare{1-X(x)}b_j\pare{t,y}-\pare{1-X(y)}b_j\pare{t,x}}dy\\
&+\int_\Omega R(x,y)\pare{\pare{1-X(x)}a_j(t,y)-X(y)b_j(t,x)}\, dy
+ e^{-t}R(\bar x, x)\pare{1-X(x)},
\end{align}
and, in a similar way, we obtain
\begin{equation}\label{eqp2}
\displaystyle \frac{\partial a_j}{\partial t}\pare{t,x}=\int_{\Omega}J\pare{x, y}\pare{X(x)a_j\pare{t,y}-X(y)a_j\pare{t,x}}\, dy
+\int_{\Omega}R\pare{x, y}\pare{X(x)b_j(t,y)-\pare{1-X(y)}a_j(t,x)}\, dy +e^{-t}J(\bar x, x)X(x).
\end{equation}
Hence, the limit is a solution to \eqref{sys1.introa.22}. 

{\it Case 2.} There exists a sub-subsequence ${\nkj}$ such that 
\begin{align}
\chi_{B_{\nkj}}\pare{\bar x}=1,\qquad\forall \nkj.
\end{align}
For this case, arguing as we did before, it is possible to prove that the limits $a_k(t,x)$ and $b_k(t,x)$ satisfy system \eqref{sys1.introb}.
\end{proof}

For what concerns the stochastic process $\pare{Y_n(t), I_n(t)}$ we can prove an analogous result to Theorem \ref{Prop:1} even in the case in which the process starts with $\delta_{\bar x}$, but now we are able to characterize the measure of the limit process $(Y(t), I(t))$ 
only when $\bar{x} \in A_n$ or $\bar{x} \in B_n$ for every $n$ (since in this case we have convergence of the densities $u_n$).
The details of this characterization can be done as in Section \ref{sect-u0L2} and are left to the reader.
Remark that the convergence of the measure holds only along subsequences.

\subsection{Asymptotic behavior of $z_n(t, x)$.}
In this subsection we look for the asymptotic behaviour as $t\to +\infty$ of $z_n(t, x)$.

\begin{proof}[Proof of Theorem \ref{teo.comport.intro}]
Let $$w_n(t,x):=z_n(t,x)-\frac{1}{\abs{\Omega}}(1-e^{-t})$$ and note that
\begin{align}
\int_{\Omega}w_n(t,x)=0.
\end{align}
To conclude the proof it is enough to show that there exists $C>0$ and $A>0$ such that, for $t$ large enough,
\begin{align}\label{gulli}
\norm{w_n(t, \cdot)}_{L^2(\Omega)}^2\leq C e^{-At}.
\end{align}

Recall the definition of $E_n(w)$ given in \eqref{ener}. Following the same strategy we used to prove Theorem \ref{teo.comport.intro0} we get that
\begin{align}
\frac{1}{2}\frac{d}{dt}\int_\Omega w_n^2(t, x)dx=-2E_n(w_n)
+e^{-t}\int_\Omega c(x, \bar x, n)w_n(t, x) dx,
\end{align}
with
$$c(x, \bar x, n):=\chi_{A_n}(\bar x)\chi_{A_n}(x)J(\bar x, x)+\chi_{B_n}(\bar x)\chi_{B_n}(x)G(\bar x, x)+\chi_{A_n}(\bar x)\chi_{B_n}(x)R(\bar x, x)+\chi_{B_n}(\bar x)\chi_{A_n}(x)R(\bar x, x).$$

By Lemma 4.1 in \cite{nosotros} it holds that there exists a constant $c_1$ (independent of $n$) such that $E_n(w_n)\geq 2c_1\norm{w_n(t, \cdot)}^2_{L^2(\Omega)}$ and therefore, by Cauchy-Schwartz's inequality, we get
\begin{align}\label{ine2}
\frac{d}{dt}\norm{w_n(t, \cdot)}^2_{L^2(\Omega)}&\leq -4c_1\norm{w_n(t, \cdot)}_{L^2(\Omega)}^2+2e^{-t}\int_\Omega c(x, \bar x, n)\abs{w_n(t, x)}dx\\
&\leq -4c_1\norm{w_n(t, \cdot)}_{L^2(\Omega)}^2+2e^{-t}\pare{c_2\norm{w_n(t, \cdot)}^2_{L^2(\Omega)}},
\end{align}
where $c_2=\pare{\norm{J}_\infty+2\norm{R}_\infty+\norm{G}_\infty}\abs{\Omega}$.
Fix $\bar t>>1$ such that $-4c_1+2c_2e^{-\bar t}\leq -2c_1$, then for every $t\geq \bar t$ we get 
\begin{align}\label{ine3}
\frac{d}{dt}\norm{w_n(t, \cdot)}^2_{L^2(\Omega)}&\leq -4c_1\norm{w_n(t, \cdot)}_{L^2(\Omega)}^2+2c_2e^{-t}.
\end{align}
By \eqref{ine3} we conclude that, 
\begin{align}\label{inpa}
\norm{w_n(t, \cdot)}^2_{L^2(\Omega)}&\leq e^{-4c_1t}\norm{w_n(\bar t, \cdot)}^2_{L^2(\Omega)}+\frac{2c_2}{4c_1-1}e^{-t}
\end{align}
and therefore we can conclude \eqref{gulli}.
\end{proof}

\begin{section}{The Dirichlet case.} \label{sect-Dirichlet}
In this final section we analyze the Dirichlet problem in which we take a sequence of partitions
$A_n$, $B_n$ of the entire space $\mathbb R^N$ such that $\mathbb R^N = A_n \cup B_n$, $A_n\cap B_n= \emptyset$ and
\begin{equation} \label{cond.sets2}
\begin{array}{l}
\bullet \  \chi_{A_n} (x) \rightharpoonup X(x), \qquad \mbox{ weakly in } L^\infty (\mathbb R^N), 
 \\[10pt]
\bullet  \chi_{B_n} (x) \rightharpoonup 1-X(x) \qquad \mbox{ weakly in } L^\infty (\mathbb R^N),\\[10pt]
\qquad \mbox{with } 0< X(x) <1.
\end{array}
\end{equation}

As for the Neumann case, at the times $\{\tau_k\}$ a particle that is at $x\in \Omega$ chooses a new site $y$, but now 
$y\in \mathbb{R}^N$, according to the kernels $J$, $R$ or $G$. The jumps from a site in $A_n$ to another site in $A_n$ are ruled by $J$,  
the jumps between $A_n$ and $B_n$ (or vice versa) are ruled by $R$ and the jumps from a site in $B_n$ to a site in $B_n$ are ruled by $G$. 
Hence, the movement of the particle obeys the same rules as before,
but now the particle is allowed to jump outside $\overline\Omega$, and, as soon as this happens, the particle is
killed and disappears from the system.
In this new model we denote by $Z_n(t)$ the position of the particle that is alive in $\overline{\Omega}$ and we suppose 
(as we did before, but this time in the whole $\mathbb{R}^N$) that we have probability kernels in our equations, that is,
\begin{align*}
\int_{\mathbb {R}^N}J(x,y)dy= 1, \quad \int_{\mathbb {R}^N}R(x,y)dy= 1,\quad \int_{\mathbb {R}^N}G(x,y)dy= 1, \qquad\forall x\in\overline\Omega.
\end{align*}

The process $\pare{Z_n(t)}_{t\geq 0}$ is a Markov process whose generator ${L}_n$ is defined on functions  $f\in C\pare{A_n}\cap C\pare{B_n}$ such that $\text{supp} f \subseteq \overline\Omega$ as
\begin{align}\label{gen.intro2}
\begin{aligned}
\hspace{-25pt}{L}_n f(x)&=\chi_{\brc{A_n\cap\overline\Omega}}\pare{x}\int_{\mathbb {R}^N}\chi_{A_n}\pare{y} J\pare{x,y}\pare{f\pare{y}-f\pare{x}}dy+\chi_{\brc{B_n\cap\overline\Omega}}\pare{x}\int_{\mathbb {R}^N}\chi_{B_n}\pare{y} G\pare{x,y}\pare{f\pare{y}-f\pare{x}}dy\\
&\quad +\chi_{\brc{A_n\cap\overline\Omega}}\pare{x}\int_{\mathbb {R}^N} \chi_{B_n}\pare{y}R\pare{x,y}\pare{f\pare{y}-f\pare{x}}dy+\chi_{\brc{B_n\cap\overline\Omega}}\pare{x}\int_{\mathbb {R}^N} \chi_{A_n}\pare{y}R\pare{x,y}\pare{f\pare{y}-f\pare{x}}dy.
\end{aligned}
\end{align}
Again the initial position $Z_n(0)$ is described in terms of a given distribution $u_0$ in $\overline\Omega$. We suppose that
\begin{align}
P\pare{Z_n\pare{0}\in E}=\int_{E} u_0(z) \, dz, \end{align}
for every measurable set $E\subseteq \overline\Omega$.

The associated evolution problem reads as
\begin{equation}\label{evol.intro2}
\left\{
\begin{array}{ll}
\displaystyle \frac{\partial u_n}{\partial t} (t,x) = {L}_n u_n(t,x), \qquad & \, t>0, \, x\in \overline{\Omega}, \\[10pt]
u_n(t,x)=0, & \, t\geq 0, \,  x\in\overline\Omega^c, \\[10pt]
u_n(0,x)=u_0(x), \qquad & x\in \Omega.
\end{array}
\right.
\end{equation}

As before we are interested in taking the limit, as $n\to +\infty$, both in the processes $Z_n(t)$ and in the associated densities $u_n(t,x)$. 
To this end we need to look at the process $Z_n(t)$ as a couple $\pare{Z_n\pare{t}, I_n\pare{t}}$. In our notation $I_n(t)$ contains explicitly the information over the set ($A_n$ or $B_n$) in which $Z_n(t)$ is located. More precisely, $I_n\pare{t}=1$ (or $2$) if the particle is in $A_n$ (or in $B_n$ respectively)
at time $t$.

The following theorem holds.

\begin{theorem} \label{teo.main.intro2}
Assume \eqref{cond.sets2} and fix $T>0$. We have that, as $n\to \infty$,
\begin{equation}\label{limite.debil.u.intro2}
\begin{array}{l}
\displaystyle u_n(t,x) \rightharpoonup u(t,x),\qquad \mbox{weakly in } L^2 ( (0,T) \times \mathbb R^N), \\[6pt]
\displaystyle \chi_{A_n} (x) u_n(t,x) \rightharpoonup a(t,x),\qquad \mbox{weakly in } L^2 ((0,T)\times \mathbb R^N), \\[6pt]
\displaystyle \chi_{B_n} (x) u_n(t,x) \rightharpoonup b(t,x),\qquad \mbox{weakly in } L^2 ((0,T) \times \mathbb R^N). 
\end{array}
\end{equation}
These limits verify
$$
u(t,x) = a(t,x) + b(t,x)
$$
and are characterized by
the fact that $(a,b)$ is the unique solution to the following system,
\begin{equation}\label{sys11.intro}
\left\{
\begin{array}{ll}
\displaystyle \frac{\partial a}{\partial t}\pare{ t,x}=\int_{\mathbb R^N}J\pare{x, y}\pare{X(x)a\pare{t,x}-X(y)a\pare{t,x}}\, dy
\\[10pt]
\displaystyle \qquad \qquad \quad 
+\int_{\mathbb R^N}R\pare{x, y}\pare{X(x)b(t,y)-\pare{1-X(y)}a(t,x)}\, dy \qquad & \, t>0,\, x\in \Omega , \\[10pt]
\displaystyle \frac{\partial b}{\partial t}\pare{t,x}=\int_{\mathbb R^N} G\pare{x, y}\cor{\pare{1-X(x)}b\pare{t,y}-\pare{1-X(y)}b\pare{t,x}}dy\\[10pt]
\qquad\qquad+\displaystyle\int_{\mathbb R^N} R(x,y)\cor{\pare{1-X(x)}a(t,y)-X(y)b(t,x)}dy\, dy
\qquad &  \, t>0,\, x\in \Omega ,\\[10pt]
a\pare{t,x}=b(t,x)=0,\quad & \, t\geq 0, \, x\in\overline\Omega^c,\\[10pt]
a\pare{0,x}=X\pare{x}u_0\pare{x}, \quad b\pare{0,x}=\pare{1-X(x)}u_0\pare{x}
\qquad & x\in \Omega.
\end{array} \right.
\end{equation}

Moreover, it holds that
the sequence of processes converges in distribution 
\begin{align}\label{convd2}
\pare{Z_n\pare{t}, I_n\pare{t}}\xrightarrow[n\to +\infty]{D}\pare{Z\pare{t}, I\pare{t}}
\end{align}
in $D\pare{[0, T], \mathbb R^N}\times D\pare{[0, T], \brc{1,2}}$, where the distribution of the limit $\pare{Z\pare{t}, I\pare{t}}$ is characterized by 
having as probability densities $a(t,x)$ and $b(t,x)$, that is,
\begin{align}
P\Big( Z\pare{t}\in  E, I(t) =1 \Big) = \int_{E} a(t,z) \, dz \quad \mbox{and} \quad 
P \Big( Z\pare{t}\in  E, I(t) =2 \Big) = \int_{E} b(t,z) \, dz ,
\end{align}
for every measurable set $ E\subseteq \mathbb R^N$.
\end{theorem}

The proof of the previous theorem follows exactly the same strategy that we used to prove Theorem \ref{teo.main.intro} as it still holds that
\begin{align}
\sup_{y\in\mathbb R^N}\abs{\int_{\Omega}V(x, y)\chi_{A_n}(x)\phi(t,x)dx-\int_{\Omega}V(x, y)X(x)\phi(t,x)dx}\xrightarrow[n\to\infty]{}0,
\end{align}
and 
\begin{align}
\sup_{x\in\overline\Omega}\abs{\int_{\mathbb R^N}\brc{\chi_{A_n}(y)-X(y)}V(x, y)dy}\xrightarrow[n\to\infty]{}0,
\end{align}
for every $V\in \brc{J, G, R}$ and $\phi:\mathbb R^N\times [0, T]\to\mathbb R$ smooth.

The fact that the limits verify the Dirichlet condition
$$a(t,x)=b(t,x)=0$$ for all $x\in \overline\Omega^c$ and $t\geq 0$ is a consequence of the second condition in \eqref{evol.intro2}.
\end{section}

\section*{Acknowledgments}

J.C.N. supported by CAPES - INCTmat grant 465591/2014-0 (Brazil). M.C.P. partially supported by CNPq grant 303253/2017-7 (Brazil). M.C. and J.D.R. partially supported by CONICET grant PIP GI No 11220150100036CO
(Argentina), PICT-2018-03183 (Argentina) and UBACyT grant 20020160100155BA (Argentina).

	\noindent\textbf{Addresses:}

{Monia Capanna and Julio D. Rossi
\hfill\break\indent
CONICET and Departamento  de Matem{\'a}tica, FCEyN,\hfill\break\indent 
Universidad de Buenos Aires, 
\hfill\break\indent  Ciudad Universitaria, Pabellon I, (1428).
Buenos Aires, Argentina.}
\hfill\break\indent
{{\tt moniacapanna@gmail.com, jrossi@dm.uba.ar}}

{Jean C. Nakasato 
	\hfill\break\indent Dpto. de Matem{\'a}tica, ICMC,
	Universidade de S\~ao Paulo, \hfill\break\indent Avenida Trabalhador S\~ao-Carlense, 400, S\~ao Carlos - SP, Brazil } \hfill\break\indent {{\tt nakasato@ime.usp.br} \hfill\break\indent {\it
		Web page: }{\tt www.ime.usp.br/$\sim$nakasato}}
	
{Marcone C. Pereira
		\hfill\break\indent Dpto. de Matem{\'a}tica Aplicada, IME,
		Universidade de S\~ao Paulo, \hfill\break\indent Rua do Mat\~ao 1010, 
		S\~ao Paulo - SP, Brazil. } \hfill\break\indent {{\tt marcone@ime.usp.br} \hfill\break\indent {\it
			Web page: }{\tt  www.ime.usp.br/$\sim$marcone}}

\end{document}